\pdfoutput=1
\documentclass[11pt,oneside]{amsart}
\usepackage[
paper=a4paper,
headsep=20pt,text={136mm,208mm},centering,includehead
]{geometry}
\usepackage{graphicx}
\usepackage[dvipsnames,usenames]{xcolor}
\usepackage[colorlinks=true,urlcolor=ForestGreen,linkcolor=ForestGreen,citecolor=ForestGreen]{hyperref}
\hypersetup{
	linkbordercolor={1 0 0}, 
	citebordercolor={0 1 0} 
}
\usepackage{wasysym}
\usepackage{amssymb,mathtools,mathrsfs,stmaryrd}
\usepackage{bbm}
\usepackage[UKenglish]{babel}
\usepackage[noabbrev]{cleveref}
\usepackage{tikz}
\usetikzlibrary{trees,decorations.pathmorphing,decorations.markings,matrix,shapes}
\usepackage[backgroundcolor=green,linecolor=green,bordercolor=green]{todonotes}
\usepackage{enumitem}
\usepackage{faktor}

\usepackage{libertine}
\usepackage[vvarbb,bigdelims,cmintegrals,libertine]{newtxmath}

\iftrue
\makeatletter
\def\@settitle{%
	\vspace*{-20pt}
	\begin{flushleft}%
		\baselineskip14\p@\relax
		\normalfont\bfseries\LARGE
		\@title
	\end{flushleft}%
}
\def\@setauthors{%
	\begingroup
	\def\thanks{\protect\thanks@warning}%
	\trivlist
	\large \@topsep30\p@\relax
	\advance\@topsep by -\baselineskip
	\item\relax
	\author@andify\authors
	\def\\{\protect\linebreak}%
	\authors
	\ifx\@empty\contribs
	\else
	,\penalty-3 \space \@setcontribs
	\@closetoccontribs
	\fi
	\normalfont
	\@setaddresses
	\endtrivlist
	\endgroup
}
\def\@setaddresses{\par
	\nobreak \begingroup\raggedright
	\small
	\def\author##1{\nobreak\addvspace\smallskipamount}%
	\def\\{\unskip, \ignorespaces}%
	\interlinepenalty\@M
	\def\address##1##2{\begingroup
		\par\addvspace\bigskipamount\noindent
		\@ifnotempty{##1}{(\ignorespaces##1\unskip) }%
		{\ignorespaces##2}\par\endgroup}%
	\def\curraddr##1##2{\begingroup
		\@ifnotempty{##2}{\nobreak\noindent\curraddrname
			\@ifnotempty{##1}{, \ignorespaces##1\unskip}\/:\space
			##2\par}\endgroup}%
	\def\email##1##2{\begingroup
		\@ifnotempty{##2}{\smallskip\nobreak\noindent E-mail address%
			\@ifnotempty{##1}{, \ignorespaces##1\unskip}\/:\space
			\ttfamily##2\par}\endgroup}%
	\def\urladdr##1##2{\begingroup
		\def~{\char`\~}%
		\@ifnotempty{##2}{\nobreak\noindent\urladdrname
			\@ifnotempty{##1}{, \ignorespaces##1\unskip}\/:\space
			\ttfamily##2\par}\endgroup}%
	\addresses
	\endgroup
	\global\let\addresses=\@empty
}
\def\@setabstracta{%
	\ifvoid\abstractbox
	\else
	\skip@25\p@ \advance\skip@-\lastskip
	\advance\skip@-\baselineskip \vskip\skip@
	\box\abstractbox
	\prevdepth\z@ 
	\vskip-15pt
	\fi
}
\renewenvironment{abstract}{%
	\ifx\maketitle\relax
	\ClassWarning{\@classname}{Abstract should precede
		\protect\maketitle\space in AMS document classes; reported}%
	\fi
	\global\setbox\abstractbox=\vtop \bgroup
	\normalfont\small
	\list{}{\labelwidth\z@
		\leftmargin0pc \rightmargin\leftmargin
		\listparindent\normalparindent \itemindent\z@
		\parsep\z@ \@plus\p@
		
	}%
	\item[\hskip\labelsep\bfseries\abstractname.]%
}{%
	\endlist\egroup
	\ifx\@setabstract\relax \@setabstracta \fi
}

\def\ps@headings{\ps@empty
	\def\@evenhead{%
		\setTrue{runhead}%
		\normalfont\scriptsize
		\rlap{\thepage}\hfill
		\def\thanks{\protect\thanks@warning}%
		\leftmark{}{}}%
	\def\@oddhead{%
		\setTrue{runhead}%
		\normalfont\scriptsize
		\def\thanks{\protect\thanks@warning}%
		\rightmark{}{}\hfill \llap{\thepage}}%
	\let\@mkboth\markboth
}\ps@headings

\def\section{\@startsection{section}{1}%
	\z@{-1.2\linespacing\@plus-.5\linespacing}{.8\linespacing}%
	{\normalfont\bfseries\Large}}
\def\subsection{\@startsection{subsection}{2}%
	\z@{-.8\linespacing\@plus-.3\linespacing}{.3\linespacing\@plus.2\linespacing}%
	{\normalfont\bfseries\large}}
\def\subsubsection{\@startsection{subsubsection}{3}%
	\z@{.7\linespacing\@plus.1\linespacing}{-1.5ex}%
	{\normalfont\bfseries}}
\def\@secnumfont{\bfseries}
\makeatother
\fi 

\makeatother

\newtheorem{theorem}{Theorem}[section]
\newtheorem*{theorem*}{Theorem}
\newtheorem*{corollary*}{}
\newtheorem{proposition}[theorem]{Proposition}
\newtheorem{corollary}[theorem]{Corollary}
\newtheorem{lemma}[theorem]{Lemma}

\newtheorem{example}[theorem]{Example}

\theoremstyle{definition}
\newtheorem{defin}[theorem]{Definition}
\newenvironment{definition}
{\pushQED{\qed}\defin}
{\popQED\enddefin}

\makeatletter
\newenvironment{statement}[1][\proofname]{\par
	\normalfont \topsep6\p@\@plus6\p@\relax
	\trivlist
	\item[\hskip\labelsep
	\bfseries
	#1\@addpunct{.}]\ignorespaces
}{%
	\endtrivlist\@endpefalse
}
\providecommand{\proofname}{Proof}
\makeatother

\numberwithin{equation}{section}

\newcommand{\dkh}{DKh}

\newcommand{\Z}{\mathbb{Z}}
\newcommand{\sg}{\mathfrak{s}}

\newcommand{\cp}{\mathfrak{p}}

\begin{document}
	
\vspace*{-40pt}
\title{A parity for \(2\)-colourable links}

\author{William Rushworth}
\address{
	Department of Mathematics and Statistics, McMaster University
}
\email{\href{mailto:will.rushworth@math.mcmaster.ca}{will.rushworth@math.mcmaster.ca}}

\def\subjclassname{\textup{2010} Mathematics Subject Classification}
\expandafter\let\csname subjclassname@1991\endcsname=\subjclassname
\expandafter\let\csname subjclassname@2000\endcsname=\subjclassname
\subjclass{57M25, 57M27, 57N70}

\keywords{link parity, virtual link concordance, Gaussian parity}

\begin{abstract}
	We introduce the \emph{\(2\)-colour parity}. It is a theory of parity for a large class of virtual links, defined using the interaction between orientations of the link components and a certain type of colouring. The \(2\)-colour parity is an extension of the Gaussian parity, to which it reduces on virtual knots. We show that the \(2\)-colour parity descends to a parity on free links. We compare the \(2\)-colour parity to other parity theories of virtual links, focusing on a theory due to Im and Park. The \(2\)-colour parity yields a strictly stronger invariant than the Im-Park parity.
	
	We introduce an invariant, the \emph{\( 2 \)-colour writhe}, that takes the form of a string of integers. The \( 2 \)-colour writhe is a concordance invariant, and so obstructs sliceness. It is also an obstruction to (\( \pm \))-amphichirality and chequerboard colourability within a concordance class.	
\end{abstract}

\maketitle

\section{Introduction}\label{1Sec:intro}
In this paper we define the \emph{\(2\)-colour parity}, a theory of parity for a large class of virtual links. A parity in the context of virtual knot theory is a designation of the classical crossings of a virtual link diagram as either \emph{even} or \emph{odd}, satisfying certain axioms. The concept of parity, due to Manturov and developed by Ilyutko, Manturov and Nikonov among others \cite{Manturov2010,Ilyutko2013}, is a powerful tool that has been used to obtain a number of results that are often difficult to obtain using other methods. While parity has been fruitful in the study of virtual knots, extensions of the concept to virtual links have been hampered by a number of defects.

Many extensions of parity to virtual links are unable to distinguish mixed crossings between the same components \cite{ImPark2013,Xu2018}. That is, given \( D_1 \) and \( D_2 \) components of a virtual link diagram, such parities declare a crossing between \( D_1 \) and \( D_2 \) as odd if and only if every crossing between \( D_1 \) and \( D_2 \) is declared as odd. It follows that the invariants extracted from these extensions of parity depend heavily on the pairwise linking numbers of the components, and are often completely determined by them. The \(2\)-colour parity does not suffer from this, and yields invariants that are strictly stronger than the pairwise linking numbers.

Manturov produced extensions of parity to virtual links that are not subject to the defect outlined above \cite{Manturov2012,Manturov2016}. However, these extensions are restricted to virtual links appearing in given cobordisms between virtual knots, or to \(2\)-component virtual links. The \(2\)-colour parity does not require any extra concordance information, and is defined for links of an arbitrary number of components.

One of the most useful parities of virtual knots is the \emph{Gaussian parity}, so-named as its definition is in terms of Gauss diagrams. Although virtual links have well-defined Gauss diagrams, the standard definition of the Gaussian parity is specific to virtual knots and does not readily extend to virtual links. The \(2\)-colour parity is an extension of the Gaussian parity to \(2\)-colourable virtual links.

Postponing the precise definition until \Cref{2Sec:definition}, a \emph{\(2\)-colouring} of a virtual link diagram is a certain colouring of it using the colours red and green; we say that a virtual link possessing such a colouring is \emph{\(2\)-colourable}. Kauffman introduced \(2\)-colourings under the name \emph{proper colourings} in \cite{Kauffman2004b}. Up to dualizing (interchanging red and green) a virtual knot diagram has only one \(2\)-colouring. Similarly, a virtual knot has only one orientation up to reversal, and one may compare the \(2\)-colouring and the orientation at classical crossings to define a parity. The construction of such a parity is naturally insensitive to colour dualizing and orientation reversal, so that a unique parity is defined. In fact, this definition recovers the Gaussian parity.

This alternate definition of the Gaussian parity has the advantage of not being specific to virtual knots; as we demonstrate in \Cref{2Sec:definition}, it can be applied to \(2\)-colourable virtual links in order to define a parity of such objects. An outline of the construction is as follows. For virtual knots, any two of the Gaussian parity, the orientation, and the \(2\)-colouring determine the third. We express this diagrammatically as
\begin{center}
	\begin{tikzpicture}
	[scale=1]
		\node[] (s0)at (0,1)  {\large parity};

		\node[] (s1)at (-1.5,0)  {\large \(2\)-colouring};

		\node[] (s2)at (1.5,0)  {\large orientation};

		\draw[<->,thick] (s0)--(s1) ;

		\draw[<->,thick] (s0)--(s2) ;

		\draw[<->,thick] (s1)--(s2) ;
\end{tikzpicture}
\end{center}
The new observation in this paper is that this relationship extends to virtual links. A virtual link has multiple inequivalent orientations and \(2\)-colourings, however, and any choice of such defines a parity. We fix the bottom-right vertex of the diagram above by working with oriented virtual links. Fixing the bottom-left vertex by picking a \(2\)-colouring, we obtain a parity: the parities obtained from two distinct \(2\)-colourings are inequivalent, in general. We use this multiplicity to define invariants by looking at the set of all \(2\)-colourings of an oriented virtual link.

While our methods are combinatorial, our results may be reformulated topologically. This and other applications of the \(2\)-colour parity will be the subject of forthcoming work.

\subsection*{Statement of results}
The main result of this paper is the introduction of the \emph{\(2\)-colour parity}. It is a strong parity on the large class of \( 2 \)-colourable oriented virtual links. Unlike previous extensions of parity, its definition does not distinguish between self- and mixed crossings, is not restricted to virtual links appearing in cobordisms, and can be computed for virtual links of an arbitrary number of components. The \(2\)-colour parity naturally extends the Gaussian parity from virtual knots to virtual links, and can be used to prove a number of topological results.

\subsubsection*{The \(2\)-colour writhe}
Let \( L \) be a \(2\)-colourable oriented virtual link. We produce a numerical invariant by looking at the set of all \(2\)-colourings of \( L \) and their associated parities. The result is a string of integers, considered up to permutation, denoted \( J^2 ( L ) \) and known as the \emph{\(2\)-colour writhe} of \( L \). We determine a number of properties of this invariant.

The \(2\)-colour writhe reduces to the odd writhe in the case of virtual knots, so that it extends that invariant to \(2\)-colourable virtual links.

\begin{statement}[\Cref{3Prop:oddwrithe}]
	Let \( K \) be a virtual knot. Then \( J^2 ( K ) = J ( K ) \), for \( J ( K ) \) the odd writhe of \( K \).
\end{statement}	

The odd writhe of virtual knots is a concordance invariant \cite{Boden2018,Rushworth2017}. The \( 2 \)-colour writhe properly extends this behaviour to virtual links.

\begin{statement}[\Cref{4Cor:writheinvariance}]
	Let \( L \) and \( L' \) be concordant \(2\)-colourable oriented virtual links. Then \( J^2 ( L ) = J^2 ( L' ) \).
\end{statement}
	
In particular, the \(2\)-colour writhe is an obstruction to sliceness. That is, it can obstruct the existence of a concordance between a virtual link and the unlink (of the appropriate number of components).

\begin{statement}[\Cref{4Cor:strongslice}]
	If \( J^2 ( L ) \) has a non-zero entry then \( L \) is not slice.
\end{statement}
	
The \(2\)-colour writhe carries information regarding (\( \pm \))-amphichirality and chequerboard colourability. Combining this with its concordance invariance, we are able to obstruct (\( \pm \))-amphichirality and chequerboard colourability within a concordance class.

\begin{statement}[\Cref{4Cor:amphiconc}]
	Let \( L \) be a \(2\)-colourable oriented virtual link such that \( J^2 ( L ) \neq -J^2 ( L ) \) (where \( -J^2 ( L ) \) denotes the string obtained by multiplying the entries of \( J^2 ( L ) \) by \( -1 \)). Then \( L \) is not concordant to a (\( \pm \))-amphichiral virtual link.
\end{statement}

\begin{statement}[\Cref{4Cor:cbconc}]
	Let \( L \) be a \(2\)-colourable oriented virtual link such that \( 0 \) does not appear in \( J^2 ( L ) \). Then \( L \) is not concordant to a chequerboard colourable virtual link.
\end{statement}
As a classical link is chequerboard colourable, it follows that the \(2\)-colour writhe provides an obstruction to a virtual link being concordant to a classical link.

In \Cref{3Prop:cbwrithe2} we show that the \(2\)-colour writhe of a chequerboard colourable virtual link may be determined from the pairwise linking numbers of the argument link. The \( 2 \)-colour writhe is a strictly stronger invariant than the pairwise linking numbers, however: in \Cref{3Sec:2colourwrithe} we exhibit a virtual link with vanishing linking numbers that is detected by the \(2\)-colour writhe.

Parity projection is a powerful construction with much utility \cite{Manturov2010}, but its application to virtual links has been restricted. The \(2\)-colour parity allows parity projection to be applied to large classes of virtual links. We demonstrate the usefulness of parity projection with respect to the \(2\)-colour parity by proving that minimal classical crossing diagram of a chequerboard colourable virtual link is itself chequerboard colourable.

\begin{statement}[\Cref{3Prop:cbminimal}]
	Let \( L \) be a chequerboard colourable oriented virtual link. A minimal classical crossing diagram of \( L \) is chequerboard colourable.
\end{statement}

\subsubsection*{Comparison with other theories}
The \(2\)-colour parity yields invariants that are strictly stronger than those associated to other extensions of parity to virtual links.

There is a very simple extension of parity that we refer to as the \emph{na\"\i ve parity}. It is defined by declaring all self-crossings of a virtual link diagram as even, and all mixed crossings as odd. The na\"\i ve parity suffers from the defect outlined above: it is unable to distinguish mixed crossings between the same components. It follows that the writhe invariant extracted from the na\"\i ve parity, known as the \emph{na\"\i ve writhe}, is simply the sum of the pairwise linking numbers, as we show in \Cref{5Subsec:naive}. In \Cref{5Thm:2colournwrithe} we show that the \(2\)-colour parity yields a strictly stronger invariant than the na\"\i ve parity (on \(2\)-colourable oriented virtual links).

The pairwise linking numbers between the components of virtual links may be odd, in contrast to those of classical links. Im and Park defined a theory of parity together with an associated writhe invariant \cite{ImPark2013}. As \cite[Figure 7]{ImPark2013} shows, their construction does not satisfy the third parity axiom (as given in \Cref{2Def:parityaxioms}), so that it does not yield a parity on arbitrary virtual links. Nevertheless, their construction does yield a parity on the restricted class of virtual links with even pairwise linking numbers. As we show in \Cref{5Prop:evensubset}, the set of such virtual links is a subset of that of \(2\)-colourable virtual links.

We refer to the resulting parity and writhe invariant of virtual links with even pairwise linking numbers as the \emph{IP parity} and the \emph{IP writhe}, respectively. We say that a classical crossing of a virtual link diagram is \emph{IP-odd} (\emph{IP-even}) if it is odd (even) with respect to the IP parity.

The IP parity is also unable to distinguish two mixed crossings between the same components. Specifically, let \( D_1 \) and \( D_2 \) be components of a virtual link diagram. A mixed crossing between \( D_1 \) and \( D_2 \) is IP-odd if and only if every mixed crossing between \( D_1 \) and \( D_2 \) is IP-odd. It follows that the contribution of mixed crossings to the IP writhe is determined by the pairwise linking numbers of its components. (In fact, there exist virtual links that are detected by the na\"\i ve writhe but not the IP writhe, as we show in \Cref{5Subsec:IP}.)

The \(2\)-colour writhe does not possess this deficiency, and, as stated above, is strictly stronger than the pairwise linking numbers. As a consequence we obtain \Cref{5Thm:2colourstronger}, which proves that the \(2\)-colour parity yields a strictly stronger invariant than the IP parity.

In subsequent work Im, Lee, and Lee \cite{ImLeeLee2014}, and Im, Kim, and Park \cite{ImKimPark2017} use the IP parity to construct polynomial invariants of virtual links with even pairwise linking numbers. In light of \Cref{5Thm:2colourstronger} it is reasonable to suspect that polynomial invariants constructed using the \(2\)-colour writhe will be stronger than these invariants.

As mentioned above, the \(2\)-colour writhe of a chequerboard colourable virtual link may be determined from the pairwise linking numbers. Nevertheless, it is still strictly stronger than both the na\"\i ve writhe and the IP writhe on such links: in \Cref{5Subsec:naive} we exhibit a chequerboard colourable virtual link that is detected by the \( 2 \)-colour writhe but is not detected by both the na\"\i ve writhe and the IP writhe.

\subsubsection*{Free links}
Free links are virtual links modulo classical crossing changes and a move known as flanking \cite{Manturov2010}. Manturov defined a parity theory of free knots (one-component free links), and extended it to free links appearing in a given concordance between free knots, and to \(2\)-component virtual links \cite{Manturov2012,Manturov2016}. The definition of the \(2\)-colour parity may be applied directly to free links.

\begin{statement}[\Cref{5Prop:freelinks} of \Cref{5Sec:comparison}]
	The \(2\)-colour parity descends to a parity on \(2\)-colourable oriented free links.
\end{statement}

The \(2\)-colour parity has the advantage that it does not require any extra concordance information, and can be determined directly from a representative of a free link. It is also defined for free links of an arbitrary number of components.

\subsection*{Plan of the paper}

\Cref{2Sec:definition} contains the definition of the \(2\)-colour parity. It also contains a complete characterisation of \(2\)-colourable virtual links, and the verification that the \(2\)-colour parity reduces to the Gaussian parity on virtual knots.

In \Cref{3Sec:2colourwrithe} we define and investigate the properties of the \(2\)-colour writhe. We also observe that the computational complexity of the \(2\)-colour writhe is quadratic in the number of link components, despite initial appearances. In the case of chequerboard colourable virtual links, we further reduce the complexity by demonstrating that the \(2\)-colour writhe of such links may be determined from the pairwise linking numbers, having first guaranteed that chequerboard links are \(2\)-colourable. (However, the \( 2 \)-colour writhe is stronger than the pairwise linking numbers, in general.)

The pairwise linking numbers are elementary concordance invariants of virtual links. Thus it is natural to ask if the \(2\)-colour writhe is a concordance invariant, in general. We provide an affirmative answer to this question in \Cref{4Sec:concordance}. We utilise a homology theory of virtual links, known as doubled Lee homology, to do so. Combining this with results of the previous sections, we illustrate that the \(2\)-colour writhe carries interesting concordance information.

In \Cref{5Sec:comparison} we compare the \(2\)-colour parity to other parity theories of virtual links. In \Cref{5Subsec:naive,,5Subsec:IP}, we demonstrate that the \(2\)-colour writhe yields a strictly stronger invariant than both the na\"\i ve parity and the IP parity. In \Cref{5Subsec:other} we briefly look at the \(2\)-colour writhe in other related contexts. \Cref{5Subsub:manturov} contains the proof that the \(2\)-colour parity descends to free links. \Cref{5Subsub:affine} compares \(2\)-colourable virtual links to those links for which the affine index polynomial may be defined. Finally, \Cref{5Subsub:xu} identifies a large class of virtual links on which an index polynomial due to Xu vanishes, but the \(2\)-colour writhe does not.

\subsubsection*{Acknowledgements} We thank Hans Boden and Andrew Nicas for their encouragement and numerous helpful conversations and comments. We thank the referee for their comments and careful reading of the paper.

\section{Definition of the \(2\)-colour parity}\label{2Sec:definition}
In this section we define the \emph{\(2\)-colour parity}. In \Cref{2Subsec:2colourings} we define the eponymous \(2\)-colourings of virtual link diagrams, and characterise the virtual links possessing such colourings. In \Cref{2Subsec:2parity} we use these colourings to define the advertised parities of links, before demonstrating in \Cref{2Subsec:gaussian} that the construction reduces to the Gaussian parity on virtual knots.

\subsection{\(2\)-colourings}\label{2Subsec:2colourings}

A \(2\)-colouring of a virtual link diagram is a certain colouring of its \emph{shadow} (the underlying flat diagram).

\begin{definition}[Shadow of a diagram]
	\label{2Def:shadow} Let \( D \) be a virtual link diagram. Denote by \( S ( D ) \) the diagram obtained from \( D \) by removing the decoration at classical crossings; we refer to the resulting double points as \emph{flat crossings}. The diagram \( S ( D ) \) is the \emph{shadow} of \( D \).
	
	Let a \emph{component} of \( S ( D ) \) be an \( S^1 \) immersed in such a way that at a flat or virtual crossing we have exactly one of the following:
	\begin{itemize}
		\item All the incident arcs are contained in the component.
		\item The arcs contained in the component are not adjacent.
		\item None of the arcs are contained in the component.
	\end{itemize}
	Thus the components of \( S ( D ) \) are in bijection with those of \( D \) and we shall not distinguish between the two.
\end{definition}

\begin{definition}[\(2\)-colouring]
	\label{2Def:2-colouring}
	Let \( D \) be a virtual link diagram. A \emph{\(2\)-colouring} of \( S ( D ) \) is a colouring of its arcs exactly one of two colours (we use red and green) such that at every flat crossing we have the following, up to rotation
	\begin{center}
		\includegraphics[scale=1]{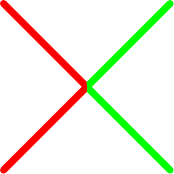}\label{Graphic:pc}
	\end{center}
	A \emph{\(2\)-colouring} of \( D \) is a \(2\)-colouring of \( S ( D ) \).
\end{definition}
Kauffman refers to such colourings as \emph{proper colourings} \cite{Kauffman2004b}.

Given a \(2\)-colouring of a virtual link diagram, we may produce another by flipping the colour configuration on one, or all, of the components.

\begin{definition}
	\label{2Def:dualizing}
	Let \( D \) be a virtual link diagram and \( \mathscr{C} \) a \(2\)-colouring of it. Denote by \( \overline{\mathscr{C}} \) be the \(2\)-colouring obtained from \( \mathscr{C} \) by \emph{global dualizing}: interchanging red and green throughout the diagram. We say that \( \overline{\mathscr{C}} \) is the \emph{global dual} of \( \mathscr{C} \).
	
	Given an arbitrary ordering of the components of \( D \), let \( {\overline{\mathscr{C}}}^i \) denote the \(2\)-colouring obtained by dualizing the colouring on the \(i\)-th component (that this yields a \(2\)-colouring is clear from the figure in \Cref{2Def:2-colouring}).
\end{definition}

A virtual link diagram which possesses a \(2\)-colouring is known as \emph{\(2\)-colourable}. Examples of \(2\)-colourable and non-\(2\)-colourable virtual links are given in \Cref{2Fig:shadowandgauss,2Fig:virtualhopflink}, respectively.

Not all virtual links are \(2\)-colourable. We conclude this section by completely characterising \(2\)-colourability. First, we define a simplified version of the Gauss diagram of a virtual link, and the complementary notion of \(2\)-colourability for such diagrams.

\begin{definition}[Simple Gauss diagram]
	\label{2Def:gauss}
	Let \( D \) be an \( n \)-component virtual link diagram and \( S ( D ) \) its shadow. Denote by \( G ( D ) \) the \emph{simple Gauss diagram of} \( D \), formed as follows:
	\begin{enumerate}
		\item Place \( n \) copies of \( S^1 \) disjoint in the plane. A copy of \( S^1 \) is known as a \emph{core circle} of \( G ( D ) \).
		\item Fix a bijection between the components of \( S ( D ) \) and the core circles of \( G ( D ) \).
		\item Arbitrarily pick a basepoint on each component of \( S ( D ) \) and on the corresponding core circle of \( G ( D ) \).
		\item Pick a component of \( S ( D ) \) and progress from the basepoint around that component (in either direction). When meeting a classical crossing label it and mark that label on the corresponding core circle of \( G ( D ) \) (virtual crossings are ignored). Continue until the basepoint is returned to.
		\item Repeat for all components of \( S ( D ) \); if a crossing is met that already has a label, use it.
		\item Add a chord linking the two incidences of each label. These chords may intersect and have their endpoints on different core circles of \( G ( D ) \).\qedhere
	\end{enumerate}
\end{definition}

\begin{definition}
	\label{2Def:altgausscolour}
	Let \( D \) be a virtual link and \( G ( D ) \) its simple Gauss diagram. The complement of the chord endpoints in the core circles is a disjoint union of intervals. A \emph{\(2\)-colouring} of \( G ( D ) \) is a colouring of these intervals exactly one of two colours such that adjacent intervals have opposite colours.
\end{definition}
Examples of a simple Gauss diagram and a \(2\)-colouring of it are given in \Cref{2Fig:shadowandgauss,2Fig:borroshadow}.

\begin{figure}
	\includegraphics[scale=0.5]{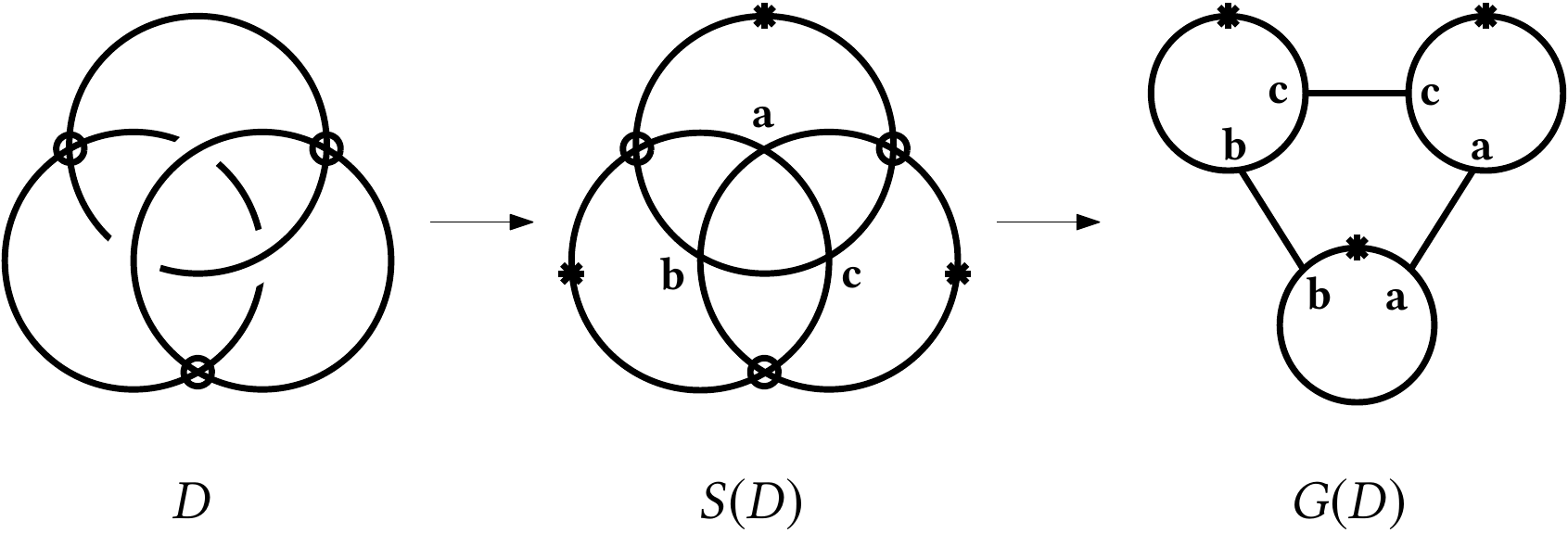}
	\caption{A virtual link diagram, its shadow, and its simple Gauss diagram.}
	\label{2Fig:shadowandgauss}
\end{figure}

\begin{figure}
	\includegraphics[scale=0.5]{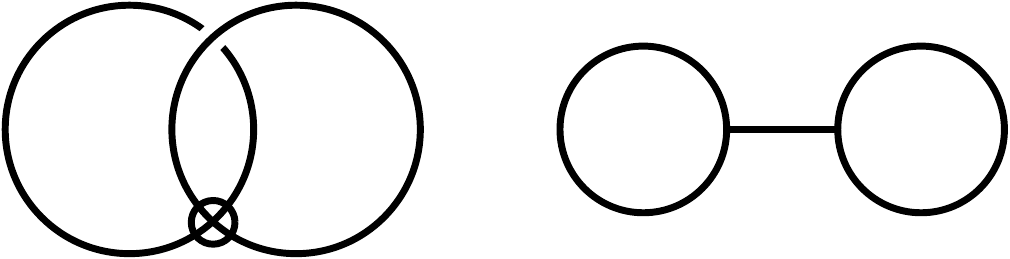}
	\caption{A virtual link that is not \(2\)-colourable, on the left, with its simple Gauss diagram, on the right.}
	\label{2Fig:virtualhopflink}
\end{figure}

The \(2\)-colourability of a virtual link diagram is completely determined by the presence of the following type of core circle in its simple Gauss diagram.

\begin{definition}
	\label{2Def:degen}
	A core circle within a simple Gauss diagram is known as \emph{degenerate} if it contains an odd number of chord endpoints.
\end{definition}

Notice that both of the core circles of the simple Gauss diagram in \Cref{2Fig:virtualhopflink} are degenerate.

The following characterisation of \(2\)-colourability is a quick exercise in combinatorics, the proof of which we include to aid the reader.
\begin{proposition}
	\label{2Thm:bijection}
	Let \( D \) be a diagram of a virtual link \( L \). Then
	\begin{equation}\label{2Eq:bij1}
	\left\lbrace~ 	\begin{matrix}
	2\text{-colourings} \\
	\text{of}~ D
	\end{matrix}
	~ \right\rbrace = \left\lbrace~ \begin{matrix}
	2\text{-colourings} \\
	\text{of}~ G ( D )
	\end{matrix} ~ \right\rbrace
	\end{equation}
	and
	\begin{equation}\label{2Eq:bij2}
	\left| \left\lbrace~ \begin{matrix}
	2\text{-colourings} \\
	\text{of}~ D
	\end{matrix}
	~ \right\rbrace \right| = \left\lbrace \begin{matrix*}[l]
	2^{ | L | }, &\text{if \( G ( D ) \) contains no degenerate circles} \\
	0, &\text{otherwise}
	\end{matrix*}
	\right.
	\end{equation}
	where \( | L | \) denotes the number of components of \( L \).
\end{proposition}

\begin{proof}
	One may readily see the bijection of \Cref{2Eq:bij1} from \Cref{2Fig:borroshadow}. We prove \Cref{2Eq:bij2} as follows. Let \( G ( D ) \) contain a degenerate core circle. On this core circle the number of connected components of the complement of the endpoints is odd, from which we deduce that it cannot be \(2\)-coloured (as the colour must change when passing an endpoint).
	
	There are two possible colour configurations of each non-degenerate core circle, and given a \(2\)-colouring of \( G ( D ) \), flipping the configuration on one core circle yields a new \(2\)-colouring. \Cref{2Eq:bij2} follows from this observation. (For more details see \cite[Theorem \(3.12\)]{Rushworth2017}.)
\end{proof}

\begin{figure}
	\includegraphics[scale=0.5]{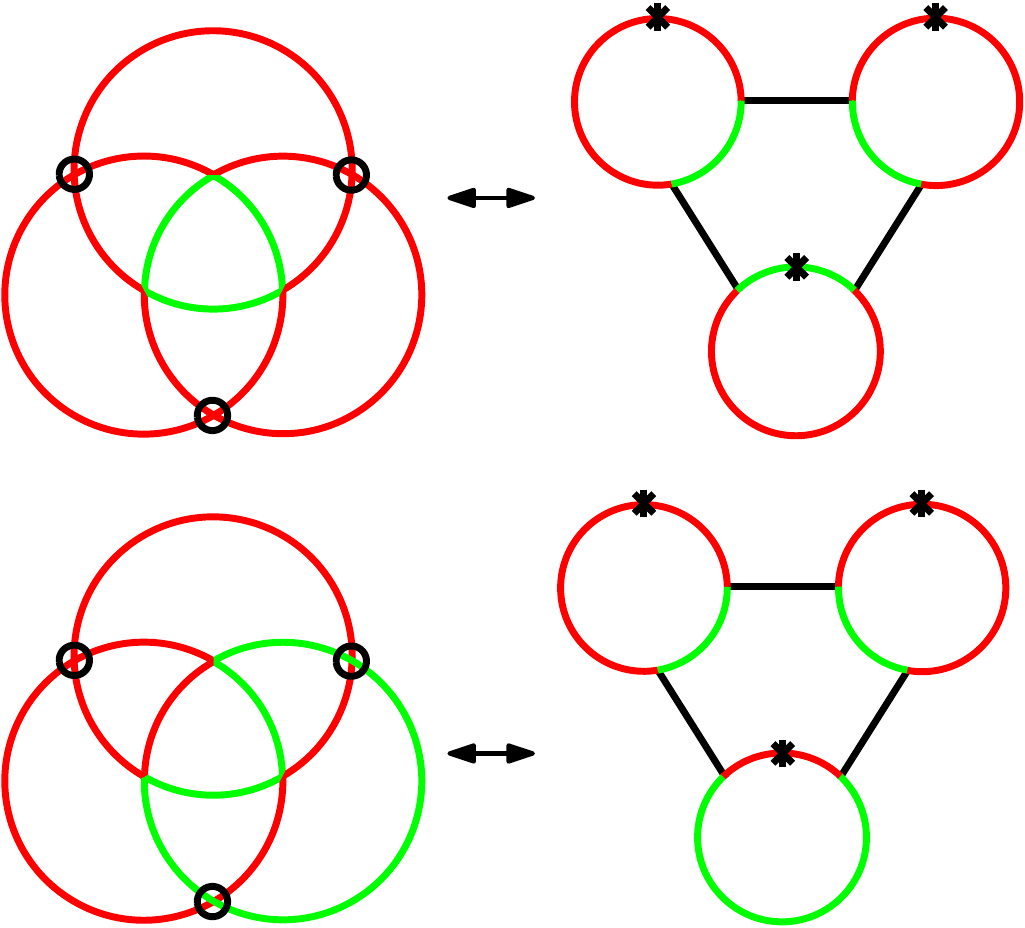}
	\caption{The bijection between \(2\)-colourings of a diagram and those of its flat Gauss diagram.}
	\label{2Fig:borroshadow}
\end{figure}

Note that \Cref{2Thm:bijection} implies that the number of \(2\)-colourings of a virtual link is diagram-independent.

\begin{corollary}\label{2Cor:vknotcolour}
	Virtual knots are \(2\)-colourable.
\end{corollary}

\begin{proof}
	The simple Gauss code of a virtual knot has exactly one core circle. Every chord endpoint must lie on this core circle so that it cannot be degenerate, as every chord has two endpoints.
\end{proof}

In \Cref{3Subsec:cbable} we demonstrate that if a virtual link is chequerboard colourable, then it is \(2\)-colourable. The converse is not true, however. Examples of \(2\)-colourable virtual links that are not chequerboard colourable are given in \Cref{2Fig:knot21,3Fig:example1}.

\subsection{The \(2\)-colour parity}\label{2Subsec:2parity}

We now employ the colourings described in \Cref{2Subsec:2colourings} to define a parity for \(2\)-colourable oriented virtual links. This is done by comparing the \(2\)-colouring to the orientation at classical crossings. First, we state the parity axioms.

\begin{definition}[Parity axioms \cite{Manturov2010}]
	\label{2Def:parityaxioms}
	Let \( \mathcal{C} \) be the category whose objects are virtual link diagrams, and morphisms are sequences of virtual Reidemeister moves. Consider the assignment of a function \( f_D \) to every object \( D \), with domain the set of classical crossings of \( D \) and codomain \( \Z_2 \). We refer to the image of a crossing under \( f_D \) as \emph{the parity} of the crossing; crossings that are mapped to \( 0 \) are \emph{even}, and those mapped to \( 1 \) are \emph{odd}.
	Such an assignment of functions is \emph{a parity} if it satisfies the following axioms:
	\begin{enumerate}[start=0]
		\item If diagrams \( D \) and \( D' \) are related by a single virtual Reidemeister move, then the parities of the crossings that are not involved in this move do not change.
		\item If \( D \) and \( D' \) are related by a Reidemeister I move that eliminates a crossing, then the parity of that crossing is even.
		\item If \( D \) and \( D' \) are related by a Reidemeister II move eliminating the crossings \( c_1 \) and \( c_2 \), then \( f_D ( c_1 ) = f_D ( c_2 ) \).
		\item If \( D \) and \( D' \) are related by a Reidemeister III move involving the crossings \( c_1 \), \( c_2 \) and \( c_3 \), then exactly one of the following holds: \( c_1 \), \( c_2 \) and \( c_3 \) are all odd, or \( c_1 \), \( c_2 \) and \( c_3 \) are all even, or exactly two of \( c_1 \), \( c_2 \) and \( c_3 \) are odd. Further, the parities of \( c_1 \), \( c_2 \) and \( c_3 \) are unchanged in \( D' \).
		\qedhere
	\end{enumerate}
\end{definition}

The above definition is of a \emph{weak} parity; if there are no Reidemeister III moves in which all of the crossings involved are odd, then the assignment is a \emph{strong} parity.

\begin{definition}[\(2\)-colour parity]
	\label{2Def:colourparity}
	Let \( D \) be an oriented virtual link diagram and \( \mathscr{C} \) a \(2\)-colouring of it. Let \( \cp_{\mathscr{C}} \) be the function from the set of classical crossings of \( D \) to \( \Z_2 \) defined as follows\footnote{we have suppressed the subscript \(D\) of \Cref{2Def:parityaxioms}}. At each classical crossing of \( D \) one may compare the orientation to the colouring of the associated flat crossing of \( S ( D ) \). The possible configurations and their image under \( \cp_{\mathscr{C}} \) are
	\begin{equation}\label{2Eq:oddeven}
		\begin{aligned}
			&\cp_{\mathscr{C}} \left( \raisebox{-11pt}{\includegraphics[scale=0.5]{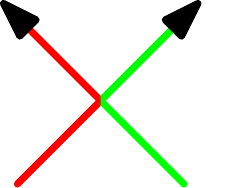}} \!\!\! \right) = \cp_{\mathscr{C}} \left(\!\!\! \raisebox{-11pt}{\includegraphics[scale=0.5]{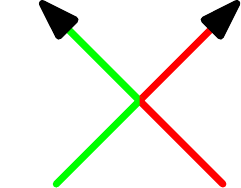}} \right) = 0 \\
			&\cp_{\mathscr{C}} \left( \raisebox{-11pt}{\includegraphics[scale=0.5]{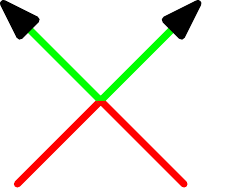}}\!\!\! \right) = \cp_{\mathscr{C}} \left( \!\!\! \raisebox{-11pt}{\includegraphics[scale=0.5]{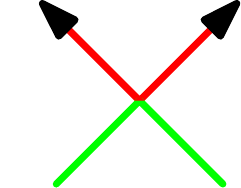}} \right) = 1
		\end{aligned}
	\end{equation}
\end{definition}

\begin{proposition}
	\label{2Prop:colourparity}
	The function \( \cp_{\mathscr{C}} \) is a strong parity (on the category of \(2\)-colourable oriented virtual link diagrams), known as the \emph{\(2\)-colour parity}.
\end{proposition}

\begin{proof}
	Axioms (\(0\)) and (\(1\)) of \Cref{2Def:parityaxioms} are easily verified. The verifications of Axioms (\( 2 \)) and (\( 3 \)) are contained in \Cref{2Fig:r2check,2Fig:r3check}; all possible colour-orientation configurations can be obtained from those depicted by reversing the orientation or dualizing the colouring on individual components.
\end{proof}

\begin{figure}
	\includegraphics[scale=0.65]{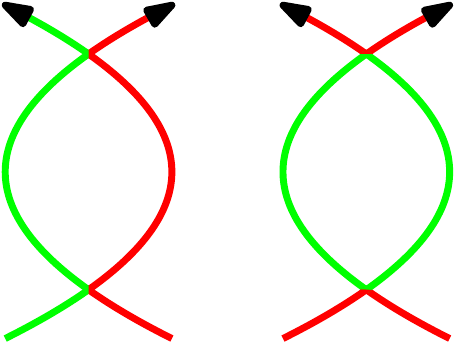}
	\caption{Verifying that the \(2\)-colour parity satisfies Axiom (\( 2 \)).}
	\label{2Fig:r2check}
\end{figure}

\begin{figure}
	\includegraphics[scale=0.65]{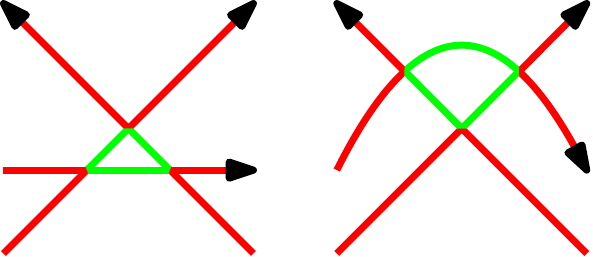}
	\caption{Verifying that the \(2\)-colour parity satisfies Axiom (\( 3 \)).}
	\label{2Fig:r3check}
\end{figure}

An advantage of the \(2\)-colour parity over other extensions of parity to virtual links is that its definition does not discriminate between and self- and mixed crossings. Despite this, self-crossings and mixed crossings behave differently: the parity of self-crossings is independent of the \(2\)-colouring used to compute it, while that of mixed crossings depends on the \(2\)-colouring.

\begin{proposition}
	\label{2Prop:selfcrossings}
	Let \( D \) be a \(2\)-colourable oriented virtual link diagram. If \( c \) is a self-crossing of \( D \) then \( \cp_{\mathscr{C}} ( c ) = \cp_{\mathscr{C}'} ( c ) \), for \( \mathscr{C} \) and \( \mathscr{C}' \) any \(2\)-colourings of \( D \).
\end{proposition}

\begin{proof}
	The \(2\)-colouring \( \mathscr{C} \) may be transformed into \( \mathscr{C}' \) by dualizing on a set of components of \( D \), denoted \( \Lambda \). Let \( c \) be a self-crossing of the component \( D_i \). If \( D_i \in \Lambda \) then the colouring is dualized on both arcs involved in \( c \). It is clear from \Cref{2Eq:oddeven} that dualizing on both arcs does not change the parity of the crossing, so that \( \cp_{\mathscr{C}} ( c ) = \cp_{\mathscr{C}'} ( c ) \).
	
	If \( D_i \notin \Lambda \) then the colouring of the arcs involved in \( c \) is unchanged, and \( \cp_{\mathscr{C}} ( c ) = \cp_{\mathscr{C}'} ( c ) \).
\end{proof}

Although the parity of a mixed crossing depends on the \(2\)-colouring, in general, it is readily observed from \Cref{2Eq:oddeven} that the parities defined by a \(2\)-colouring and its global dual are equivalent i.e.\ that they take the same value on every classical crossing (both self and mixed). Therefore a virtual link \( L \) has at most \( 2^{|L|-1} \) inequivalent parities. We use this observation frequently, via the following definition.

\begin{definition}
	\label{2Def:generatingset}
	Let \( D \) be a diagram of an oriented virtual link \( L \). From every pair of \(2\)-colourings of \( D \), \( \left(  \mathscr{C}, \mathscr{C}' \right) \), such that \( \mathscr{C}' = \overline{\mathscr{C}} \), pick one. Let
	\begin{equation*}
	\lbrace \mathscr{C}_1, \mathscr{C}_2, \ldots, \mathscr{C}_{2^{|L|-1}} \rbrace
	\end{equation*}
	be the resulting set of \(2\)-colourings. Such a set is known as a \emph{generating set of \(2\)-colourings of \(D\)}.
\end{definition}

\subsection{Relationship to the Gaussian parity}\label{2Subsec:gaussian}

\Cref{2Thm:bijection} implies that a virtual knot has exactly two \(2\)-colourings. These \(2\)-colourings are global duals of one another, however, and owing to the fact that a virtual knot has a unique orientation up to reversal, there is a unique \(2\)-colour parity for virtual knots. In this section we verify that this parity recovers the Gaussian parity.

\begin{definition}[Gaussian parity \cite{Manturov2010}]
	Let \( D \) be a virtual knot diagram. A classical crossing \( c \) of \( D \) is \emph{G-even} if one passes an even number of chord endpoints when travelling between the two endpoints of the chord associated to \( c \) (in either direction). A crossing that is not G-even is \emph{G-odd}. This declaration defines a parity, known as the \emph{Gaussian parity}.
\end{definition}
For example, both of the classical crossings of the virtual knot depicted in \Cref{2Fig:knot21} are G-odd. Note that the Gaussian parity does not depend on the orientation of the virtual knot. (The Gaussian parity can also be expressed in terms of quandles \cite{Manturov2017}.)

\begin{figure}
	\includegraphics[scale=0.75]{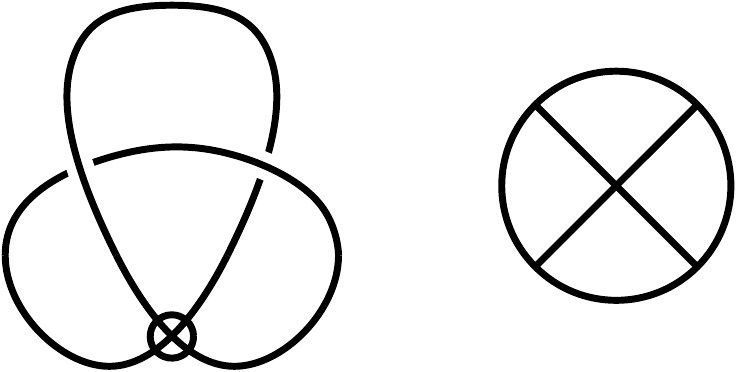}
	\caption{The virtual knot \(2.1\), on the left, and its simple Gauss code, on the right.}
	\label{2Fig:knot21}
\end{figure}

The following is a reformulation of \cite[Proposition \(4.11\)]{Rushworth2017}.

\begin{proposition}
	\label{2Prop:Gparityrelation}
	The \(2\)-colour parity is equivalent to the Gaussian parity on virtual knots.
\end{proposition}

\begin{proof}	
	We show that a classical crossing of a virtual knot diagram \( D \) is G-odd if and only if it is odd with respect to the \(2\)-colour parity.
	
	\noindent(\( \Rightarrow \)): Let \(c\) denote a G-odd classical crossing of \(D\). Leaving the crossing from either of the outgoing arcs we must return to a specified incoming arc. Between leaving and returning we have passed through an odd number of classical crossings (that are not \(c\)). Thus the incoming arc must be coloured the opposite colour to the outgoing, and \(c\) must be as follows (up to dualizing)
	\begin{center}
		\includegraphics{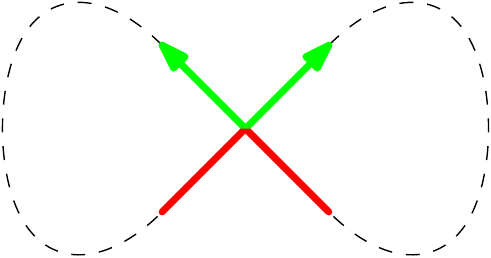}
	\end{center}
	
	\noindent(\( \Leftarrow \)): Let \(c\) denote a classical crossing of \(D\) that is odd with respect to the \(2\)-colour parity. Then the colouring at \(c\) must be as depicted above (up to dualizing). Again, leaving \(c\) from either outgoing arc and returning at the specified incoming arc, we see that, as the colours of the arcs are opposite, an odd number of classical crossings must have been passed, and \(c\) is G-odd.
\end{proof}
In light of this relationship one may interpret the \(2\)-colour parity as an extension of the Gaussian parity to \(2\)-colourable virtual links. In \Cref{4Sec:concordance} we further justify this, by showing that the \(2\)-colour parity yields a concordance invariant, replicating the concordance invariance of the odd writhe of virtual knots.

\section{The \(2\)-colour writhe}\label{3Sec:2colourwrithe}

Any parity naturally defines an integer-valued invariant of virtual knots, via a signed count of the odd crossings. In this section we use the \(2\)-colour parity to define a similar invariant of \(2\)-colourable virtual links. Rather than a single integer, however, we use a generating set of \(2\)-colourings to obtain a string of integers (the length of which depends on the number of components of the link), known as the \emph{\(2\)-colour writhe}.

\Cref{3Subsec:2writhedefinition} contains the definition of the invariant. Initially, the computational complexity of the \(2\)-colour writhe appears to be exponential in the number of components of the virtual link; in \Cref{3Subsec:complexity} we reduce this to quadratic dependence.

In \Cref{3Subsec:cbable} we demonstrate that the \(2\)-colour writhe of a chequerboard colourable virtual link may be determined from the pairwise linking numbers of its components. (Nevertheless, it remains strictly stronger than the writhe invariants associated to the na\"\i ve parity and the IP parity on such links.)

\subsection{Definition}\label{3Subsec:2writhedefinition}

The definition of the \(2\)-colour writhe follows that of the odd writhe, given in \cite{Kauffman2004b} (we demonstrate in \Cref{3Prop:oddwrithe} that it reduces to the odd writhe on virtual knots, in fact). Each \(2\)-colouring of a virtual link diagram defines a writhe, and we declare the string of all such writhes to be the \(2\)-colour writhe of the virtual link represented.

\begin{definition}[\(2\)-colour writhe]
	\label{3Def:2colourwrithe}
	Let \( D \) be a diagram of an oriented virtual link \( L \), and
		\begin{equation*}
			\lbrace \mathscr{C}_1, \mathscr{C}_2, \ldots, \mathscr{C}_{2^{|L|-1}} \rbrace
		\end{equation*}
	a generating set of \(2\)-colourings. Given a \(2\)-colouring \( \mathscr{C}_i \) define the quantity \( J_{\mathscr{C}_i} ( D ) \) as
		\begin{equation}\label{3Eq:writhe}
			J_{\mathscr{C}_i} ( D ) \coloneqq \sum_{\cp_{\mathscr{C}_i} ( c ) = 1 } \text{sign} ( c ).
		\end{equation}
	That is, it is the sum of the signs of the odd crossings (with respect to the parity \( \cp_{\mathscr{C}_i} \)).
	
	Define the \emph{\(2\)-colour writhe of \( D \)} to be
		\begin{equation*}
			J^2 ( D ) \coloneqq \left( J_{\mathscr{C}_1} ( D ), \ldots, J_{\mathscr{C}_{2^{|L|-1}}} ( D ) \right) \in \Z^{2^{|L|-1}}
		\end{equation*}
	considered up to permutation of the entries. That \( J^2 ( D ) \) is independent of the choice of generating set is clear from the observation, made earlier, that the parities associated to globally dual \(2\)-colourings are equivalent. It follows from the parity axioms that the quantities \( J_{\mathscr{C}_i} ( D ) \) are invariant under the virtual Reidemeister moves, and we may define \emph{\(2\)-colour writhe of \( L \)} as \( J^2 ( L ) \coloneqq J^2 ( D ) \).
\end{definition}

\begin{figure}
	\includegraphics[scale=0.65]{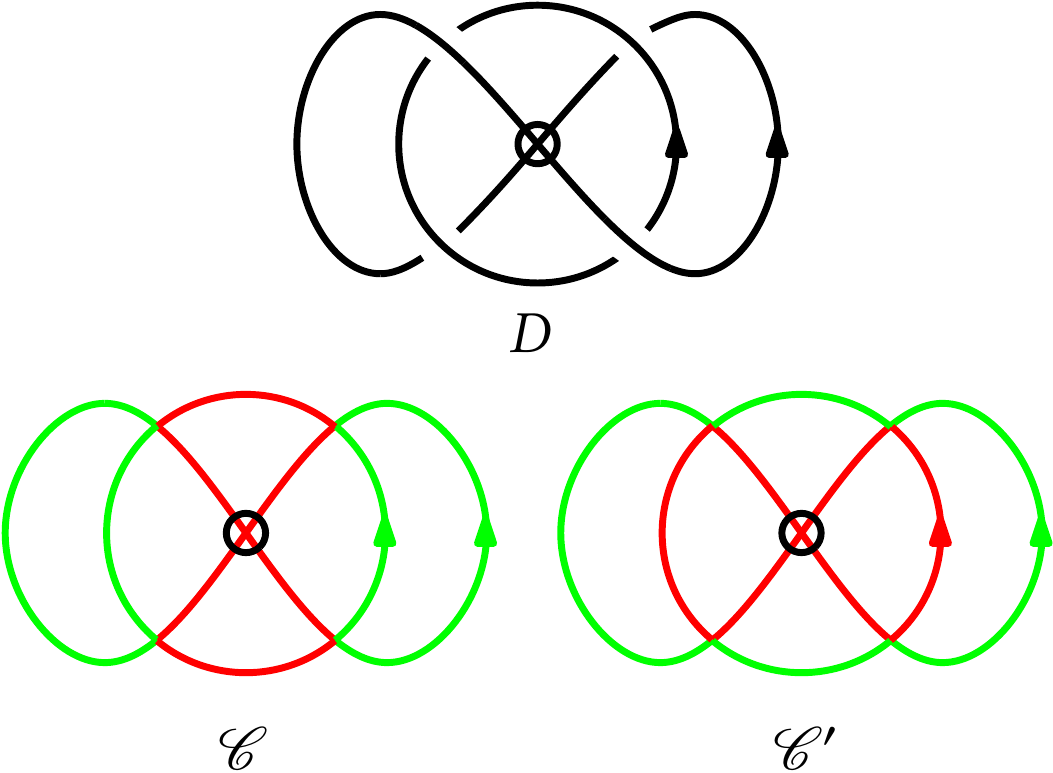}
	\caption{A two-component oriented virtual link diagram and a generating set of \(2\)-colourings.}
	\label{3Fig:example1}
\end{figure}

\begin{example}\label{3Ex:example1}
	Let \( L \) be the two-component oriented virtual link whose diagram \( D \) appears in \Cref{3Fig:example1}: \( \lbrace \mathscr{C}, \mathscr{C}' \rbrace \) is a generating set of \(2\)-colourings, and \( J_{\mathscr{C}} ( D ) = 2 \), \( J_{\mathscr{C}'} ( D ) = -2 \), so that \( J^2 ( L ) = ( 2, -2 ) \).
\end{example}

Notice that \Cref{2Prop:selfcrossings} allows us to define the following restricted invariant.

\begin{definition}
	\label{3Def:2colourselfwrithe}
	Let \( D \) be a diagram of a \(2\)-colourable oriented virtual link \( L \), and \( \mathscr{C} \) a \(2\)-colouring of \( D \). Define the quantity
	\begin{equation*}
		J^2_S ( D ) \coloneqq \sum_{\substack{c ~\text{self-crossing} \\ \cp_{\mathscr{C}} ( c ) = 1 }} \text{sign} ( c )
	\end{equation*}
	That is, \( J^2_S ( D ) \) is the sum of the self-crossings of \( D \) that are odd with respect to \( \cp_{\mathscr{C}} \). It follows from \Cref{2Prop:selfcrossings} that \( J^2_S ( D ) \) is well-defined. Its invariance under the virtual Reidemeister moves follows from the fact that the \(2\)-colour parity satisfies the parity axioms, and we may define \( J^2_S ( L ) \coloneqq J^2_S ( D ) \).
\end{definition}
Note that \( J^2_{S} ( L ) \) is a single integer, while \( J^2 ( L ) \) is a string of integers.

The remainder of this work is concerned with determining a number of properties of the \(2\)-colour writhe and putting them to use: we show that it is an obstruction to (\( \pm \))-amphichirality, chequerboard colourability (and hence classicality), and that is a concordance invariant (so that it obstructs sliceness, for example).

First, we verify that the \(2\)-colour writhe reduces to the odd writhe on virtual knots.

\begin{proposition}\label{3Prop:oddwrithe}
	Let \( K \) be a virtual knot. Then \( J^2 ( K ) = J ( K ) \) (for \( J ( K ) \) the odd writhe of \( K \)).
\end{proposition}

\begin{proof}
	By \Cref{2Thm:bijection} a virtual knot has two \(2\)-colourings, and picking either one of them yields a generating set. As demonstrated in \Cref{2Prop:Gparityrelation}, the \(2\)-colour parity associated to either \(2\)-colouring is equivalent to the Gaussian parity, from which the proposition follows.
\end{proof}

Next, we demonstrate that the \(2\)-colour writhe is an obstruction to (\( \pm \))-amphichirality. There are two distinct types of mirror image in virtual knot theory. For the purposes of this article we shall not distinguish between them, regarding a virtual link as amphichiral if either mirror image preserves the link type, as in the following definition. The prefix \( (\pm) \) denotes if the mirror image preserves or reverses orientation.
\begin{definition}[Mirror image, amphichirality] Let \( D \) be an unoriented virtual link diagram. The \emph{vertical mirror image} of \( D \) is obtained by performing a crossing change at every classical crossing of \( D \). The \emph{horizontal mirror image} of \( D \) is obtained by considering \( D \) as lying to one side of the \( x = 0 \) axis in the plane, and reflecting about this axis. By \emph{a mirror image} of \( D \) we mean either the vertical or horizontal mirror image.

Given an orientation of \( D \), a \emph{positive mirror image} of \( D \) is obtained by constructing a mirror image of \( D \) and preserving the orientation on the components of \( D \). A \emph{negative mirror image} of \( D \) is obtained by constructing a mirror image of \( D \) and reversing the orientation on the components of \( D \).

Let \( D \) represent the virtual link \( L \). We say that \( L \) is \emph{\((+)\)-amphichiral} if a positive mirror image of \( D \) also represents \( L \). We say that \( L \) is \emph{\((-)\)-amphichiral} if a negative mirror image of \( D \) also represents \( L \). Finally, we say that \( L \) is \emph{(\( \pm \))-amphichiral} if it is \((+)\)-amphichiral or \((-)\)-amphichiral.
\end{definition} 

Given a \(2\)-colourable oriented virtual link \( L \), denote by \( -J^2 ( L ) \) the string obtained by multiplying the entries of \( J^2 ( L ) \) by \( -1 \).

\begin{proposition}\label{3Prop:amphichiral}
	Let \( L \) be a \(2\)-colourable oriented virtual link. If \( L \) is (\( \pm \))-amphichiral then \( J^2 ( L ) = - J^2 ( L ) \).
\end{proposition}

\begin{proof}
	Let \( D \) be a diagram of \( L \). Denote by \( m_+ ( D ) \) a positive mirror image of \( D \). There is a bijection between the \(2\)-colourings of \( D \) and \( m_+ ( D ) \): this bijection is the identity in the case of the vertical mirror image, and is induced by reflection in the case of the horizontal mirror image. In either case, given \( \mathscr{C} \) a \(2\)-colouring of \( D \) denote by \( m_+ ( \mathscr{C} ) \) the associated \( 2 \)-colouring of \( m_+ ( D ) \). There is also a bijection between the classical crossings of \( D \) and those of \( m_+ ( D ) \). As a positive mirror image preserves the orientation of the components of \( D \), it follows that a crossing of \( D \) is odd with respect to \( \cp_{\mathscr{C}} \) if and only if the associated crossing of \( m_+ ( D ) \) is odd with respect to \( \cp_{m_+ ( \mathscr{C} )} \).

	Combining the above observation with the fact that taking the vertical or the horizontal mirror image changes the sign of all classical crossings, it is clear from \Cref{3Eq:writhe} that
		\begin{equation}\label{3Eq:amphi}
			J_{\mathscr{C}} ( D ) = - J_{ m_+ ( \mathscr{C} ) } ( m_+ ( D ) ).
		\end{equation}
	If \( L \) is (\( + \))-amphichiral then \( D \) and \( m_+ ( D ) \) both represent \( L \) and
		\begin{equation*}
			J^2 ( D ) = J^2 ( m_+ ( D ) ) = J^2 ( L )
		\end{equation*}
	which, combined with \Cref{3Eq:amphi}, completes the proof in the (\( + \))-amphichiral case. The (\( - \))-amphichiral case follows from the observation that global orientation reversal has no affect on the \(2\)-colour writhe.
\end{proof}

Notice that, unlike the odd writhe of virtual knots, the \(2\)-colour writhe of (\( \pm \))-amphichiral virtual links is not forced to be zero, so that it can detect (\( \pm \))-amphichiral virtual links that are not slice.

\subsection{Reduction of computational complexity}\label{3Subsec:complexity}
Recall from \Cref{2Thm:bijection} that the number of \(2\)-colourings of a link grows exponentially with the number of components. This appears to make computation of the \(2\)-colour writhe intensive for links with of components. However, owing to the fact that a classical mixed crossing involves exactly two components, we deduce that the complexity is in fact quadratic.

First we set up a bijection between the set of \(2\)-colourings of an \(n\)-component link and the set \( \lbrace 0,1 \rbrace^{\times n} \).

\begin{definition}\label{3Def:basebijection}
	Let \( D \) be a diagram of a \(2\)-colourable oriented virtual link with \( |D| = n \), together with an arbitrary ordering of its components. Pick a \(2\)-colouring of \( D \), \( \mathscr{B} \), and declare it as the \emph{base \(2\)-colouring}. There is a bijection between \( \lbrace 0,1 \rbrace^{\times n} \) and the set of \( 2 \)-colourings of \( D \) defined as follows.
	
	Let \( \mathscr{B} \) be identified with \( (0, \ldots , 0 ) \in \lbrace 0,1 \rbrace^{\times n} \). Denote by \( v_i \in \lbrace 0,1 \rbrace^{\times n} \) the string with a \(1\) in the \(i\)-th position and \( 0 \) elsewhere, \( v_{i,j} \in \lbrace 0,1 \rbrace^{\times n} \) the string a with \(1\) in the \(i\)-th and \(j\)-th position and \(0\) elsewhere, and so forth. Identify \( v_i \in \lbrace 0,1 \rbrace^{\times n} \) with \( \overline{\mathscr{B}}^i \) (the \(2\)-colouring obtained by dualizing \( \mathscr{B} \) on the \(i\)-th component), \( v_{i,j} \in \lbrace 0,1 \rbrace^{\times n} \) with \( \overline{\mathscr{B}}^{i,j} \), and so forth.
	
	That every \(2\)-colouring of \( D \) is represented by an element of \( \lbrace 0,1 \rbrace^{\times n} \) is clear from the cardinality of the sets (recall from \Cref{2Thm:bijection} that \( D \) has \( 2^n \) \(2\)-colourings).
\end{definition}

Note that this bijection depends on the choice of base \(2\)-colouring. In what follows we shall freely interchange between elements of \( \lbrace 0,1 \rbrace^{\times n} \) and \( 2 \)-colourings, taking the bijection as understood.

\begin{definition}\label{3Def:nicegenerators}
	Let \( D \) be a diagram of a \(2\)-colourable oriented virtual link with \( |D| = n \), together with an arbitrary ordering of its components and a base \(2\)-colouring \( \mathscr{B} \). Given a \(2\)-colouring \( v = (e_1, e_2, \ldots, e_n ) \in \lbrace 0,1 \rbrace^{\times n} \), let \( | v | = \sum e_i  \). Define the subset \[ W^n_i = \lbrace v \in \lbrace 0,1 \rbrace^{\times n} ~|~ | v | = i \rbrace. \]
	Denote by \( \mathcal{G} \) the set of \(2\)-colourings of \(D\) defined as follows. If \( n \) is odd, then
		\begin{equation*}
				\mathcal{G} \coloneqq \bigcup_{i=1}^{ \lfloor \frac{n}{2} \rfloor } W^n_i .
		\end{equation*}
	If \( n \) is even and \( v \in W^n_{\frac{n}{2}} \), denote by \( \overline{v} \in W^n_{\frac{n}{2}} \) the string obtained by sending \( 0 \) to \( 1 \) and vice versa. Then
		\begin{equation*}
			\mathcal{G} \coloneqq \left( \bigcup_{i=1}^{ \lfloor \frac{n}{2} \rfloor } W^n_i \right) \cup \widetilde{W^n_{\frac{n}{2}}}
		\end{equation*}
	where \( \widetilde{W^n_{\frac{n}{2}}} \) is defined
		\begin{equation*}
			\widetilde{W^n_{\frac{n}{2}}} \coloneqq \lbrace v \in W^n_{\frac{n}{2}} ~|~ \text{\(v\) preceeds \( \overline{v} \) in the dictionary order} \rbrace . \qedhere
		\end{equation*}
\end{definition}

An example of this generating set in the case \( n = 4 \) is given in \Cref{3Fig:g4}.

\begin{figure}
	\begin{equation*}
		\begin{tabular}{c c c c c c c c c c c}
		& & & & & & \(\boxed{0011}\) & & & & \\
		& & & & \(\boxed{0001}\) & & \(\boxed{0101}\) & &\( 0111\) & & \\
		& & \(\boxed{0000}\) & & \(\boxed{0010}\) & & \(\boxed{0110}\) & & \(1011\) & & \(1111\)\\
		& & & & \(\boxed{0100}\) & & \(1001\) & & \(1101\) & & \\
		& & & & \(\boxed{1000}\) & & \(1010\) & & \(1110\) & & \\
		& & & & & & \(1100\) & & & & \\
		& & \rotatebox{-90}{\(=\)} & & \rotatebox{-90}{\(=\)} & & \rotatebox{-90}{\(=\)} & & \rotatebox{-90}{\(=\)} & & \rotatebox{-90}{\(=\)} \\
		\\
		\(\lbrace 0,1 \rbrace^{\times 4}\) & \(=\) & \(W^4_{0}\) & \(\cup \)& \(W^4_{1}\) & \(\cup\) & \(W^4_{2}\) & \(\cup\) &\( W^4_{3}\) & \(\cup\) & \(W^4_{4}\)
		\end{tabular}	
	\end{equation*}
	\caption{The set \( \lbrace 0,1 \rbrace^{\times 4} \) written as the union of \( W^4_{0} \) to \( W^4_{4} \). The generating set \( \mathcal{G} \) is represented by the boxed elements. The affect of global dualization is to send \( \mathcal{G} \) to \( \lbrace 0,1 \rbrace^{\times 4} \setminus \mathcal{G} \).}
	\label{3Fig:g4}
\end{figure}

\begin{proposition}\label{3Prop:nicegenerators}
	Let \( D \) be a diagram of a \(2\)-colourable oriented virtual link with \( |D| = n \), taken with an arbitrary ordering of its components and a base colouring \( \mathscr{B} \). The set \( \mathcal{G} \) is a generating set of \(2\)-colourings.
\end{proposition}

\begin{proof}
	It is clear that \( \lbrace 0,1 \rbrace^{\times n} = \bigcup_{i=1}^n W^n_i \), so that we are required to show every \( W^n_i \) is a subset of \( \mathcal{G} \), or can be obtained by dualizing elements of \( \mathcal{G} \). The affect of global dualization on a \(2\)-colouring identified with the string \( ( e_1, e_2, \ldots , e_n ) \) is to send \(0\)'s to \(1\)'s and \(1\)'s to \(0\)'s. From this observation it is clear that \( W^n_i \) is sent to \( W^n_{n-i} \) under dualization so that \( \mathcal{G} \) is sent to \( \lbrace 0,1 \rbrace^{\times n} \setminus \mathcal{G} \), and that the proposition holds for \( n \) odd.
	
	For \(n\) even, it remains to show that \( W^n_{\frac{n}{2}} \) is the union of \( \widetilde{W^n_{\frac{n}{2}}} \) and the set obtained by dualizing all of the elements of \( \widetilde{W^n_{\frac{n}{2}}} \). This last statement follows from the fact that there are no elements \( v \in W^n_{\frac{n}{2}} \) with \( v = \overline{v} \), and that \( \left| \widetilde{W^n_{\frac{n}{2}}} \right| = \dfrac{1}{2} \left| W^n_{\frac{n}{2}} \right| \). Therefore \( \mathcal{G} \) is sent to \( \lbrace 0,1 \rbrace^{\times n} \setminus \mathcal{G} \), as required.
\end{proof}

With this generating set in place we are able to significantly reduce the number of individual computations required to determine the \(2\)-colour writhe.

Let \( D \) be a diagram of a \(2\)-colourable oriented virtual link with \( |D| = n \), taken with an arbitrary ordering of its components, and a base colouring \( \mathscr{B} \). We shall use this diagram for the remainder of the section. Define the sets
\begin{equation*}\label{3Eq:setdefinitions}
	\begin{aligned}
		C^i &\coloneqq \lbrace \text{classical crossings of \(D\) between the \(i\)-th component and another component} \rbrace \\
		\mathscr{B}^e &\coloneqq \lbrace \text{classical crossings of}~ D ~\text{that are even with respect to}~  \mathfrak{p}_{\mathscr{B}} \rbrace \\
		\mathscr{B}^o &\coloneqq \lbrace \text{classical crossings of}~ D ~\text{that are odd with respect to}~  \mathfrak{p}_{\mathscr{B}} \rbrace.
	\end{aligned}
\end{equation*}
Recall from \Cref{3Def:basebijection} that \( v_i \) denotes the \(2\)-colouring obtained from \( \mathscr{B} \) by dualizing the colouring on the \( i \)-th component; we use the shorthand \( J_i ( D ) = J_{v_i} ( D ) \), \( J_{i,j} ( D ) = J_{v_{i,j}} ( D ) \), and so on. It follows from \Cref{3Def:2colourwrithe} that
	\begin{equation}\label{3Eq:1dual}
		J_i ( D ) = \sum_{c \in C^i \cap \mathscr{B}^e } \text{sign}( c ) - \sum_{c \in C^i \cap \mathscr{B}^o } \text{sign}( c ) + J_{\mathscr{B}}.
	\end{equation}
That is, to \( J_{\mathscr{B}} \) we the add the signs of the mixed crossings lying on the \(i\)-th component that are even with respect to \( \cp_{\mathscr{B}} \), and minus the signs of those crossings that are odd (as they are exactly the crossings that change parity when \( \mathscr{B} \) is dualized on the \(i\)-th component). By a very similar argument we observe that
	\begin{equation}\label{3Eq:2dual}
		\begin{aligned}
			J_{i,j} ( D ) &= \Sigma \left( C^i \cap \mathscr{B}^e \right) - \Sigma \left( C^i \cap \mathscr{B}^o \right) + \Sigma \left( C^j \cap \mathscr{B}^e \right) - \Sigma \left( C^j \cap \mathscr{B}^o \right) \\
			&\quad+ 2 \left( \Sigma \left( C^i \cap C^j \cap \mathscr{B}^o \right) \right) 
			- 2 \left( \Sigma \left( C^i \cap C^j \cap \mathscr{B}^e \right) \right) + J_{\mathscr{B}} ( D ) \\
			&= J_i (D) + J_j (D) + 2 \left( \Sigma \left( C^i \cap C^j \cap \mathscr{B}^o \right) \right) - 2 \left( \Sigma \left( C^i \cap C^j \cap \mathscr{B}^e \right) \right) - J_{\mathscr{B}} (D)
		\end{aligned}
	\end{equation}
where we have used the shorthand \[ \Sigma \left( C^i \cap \mathscr{B}^e \right) = \sum_{c \in C^i \cap \mathscr{B}^e } \text{sign}(c) \] and so forth.

In general, the contributions to the \(2\)-colour writhe are given by the following formula.

\begin{theorem}\label{3Thm:induction}
	Let \( v = (e_1, e_2, \ldots , e_n ) \in \mathcal{G} \) be a \(2\)-colouring with \( 2 < |v| = m \). Let \( \lbrace p_1, p_2, \ldots , p_m \rbrace \) be the set of indices such that \( e_{p_k} = 1 \) in \( v \). Then
		\begin{equation}\label{3Eq:induction}
			J_v ( D ) = \sum_{k=1}^{m} \left( \sum_{l>k}^{m} J_{p_k, p_l} \right) - (m-2) \sum_{s=1}^m J_{p_s} + \dfrac{1}{2} (m-1)(m-2) J_{\mathscr{B}}
		\end{equation}
\end{theorem}

\begin{proof}
	We proceed by induction. The base case, \( |v| = 3 \), may be checked by hand. We obtain
		\begin{equation*}
			\begin{aligned}
				J_{i,j,k} ( D ) &= \Sigma \left( C^i \cap \mathscr{B}^e \right) - \Sigma \left( C^i\cap \mathscr{B}^o \right) + \Sigma \left( C^j \cap \mathscr{B}^e \right) - \Sigma \left( C^j \cap \mathscr{B}^o \right) \\ 
				&\quad + \Sigma \left (C^k \cap \mathscr{B}^e \right) - \Sigma \left (C^k \cap \mathscr{B}^o \right) + 2 \Sigma \left (C^i \cap C^j \cap \mathscr{B}^o \right) \\
				&\quad - 2 \Sigma \left( C^i \cap C^j \cap \mathscr{B}^e \right) + 2 \Sigma \left( C^i \cap C^k \cap \mathscr{B}^o \right) - 2 \Sigma \left( C^i \cap C^k \cap \mathscr{B}^e \right) \\
				&\qquad + 2 \Sigma \left( C^j \cap C^k \cap \mathscr{B}^o \right) - 2 \Sigma \left( C^j \cap C^k \cap \mathscr{B}^e \right) + J_{\mathscr{B}} ( D ) \\
				&= J_{i,j} ( D ) + J_{i,k} ( D ) + J_{j,k} ( D ) - J_i ( D ) - J_j ( D ) - J_k ( D ) + J_{\mathscr{B}} ( D )
			\end{aligned}
		\end{equation*}
	as required.
	
	Assume that the proposition holds for \( |v|=i \). Let \( v' \) be obtained from \( v \) by dualizing on the \(r\)-th component, so that \( | v' | = i +1 \). It follows from \Cref{3Def:2colourwrithe} that
		\begin{equation*}
			\begin{aligned}
				J_{v'} ( D ) &= \sum^{i+1}_{t=1} \left( \Sigma \left( C^t \cap \mathscr{B}^e \right) - \Sigma \left( C^t\cap \mathscr{B}^o \right) \right)\\ 
				&\quad + 2 \left( \sum^{i+1}_{s=1} \left( \sum_{l>s}^{i+1} \left( \Sigma \left( C^l \cap C^s \cap \mathscr{B}^o \right) - \Sigma \left( C^l \cap C^s \cap \mathscr{B}^e \right) \right) \right) \right) + J_{\mathscr{B}} ( D )
			\end{aligned}
		\end{equation*}
	and
		\begin{equation*}
			\begin{aligned}
					J_{v} ( D ) &= \sum^{i}_{t=1} \left( \Sigma \left( C^t \cap \mathscr{B}^e \right) - \Sigma \left( C^t\cap \mathscr{B}^o \right) \right)\\ 
				&\quad + 2 \left( \sum^{i}_{s=1} \left( \sum_{l>s}^{i} \left( \Sigma \left( C^l \cap C^s \cap \mathscr{B}^o \right) - \Sigma \left( C^l \cap C^s \cap \mathscr{B}^e \right) \right) \right) \right) + J_{\mathscr{B}} ( D )
				\end{aligned}
		\end{equation*}
	where we have fixed \( i + 1 = r \) in the ordering of the components of \( D \). Then
	
	\begin{align}\label{3Eq:difference}
			\notag J_{v'} ( D ) - J_{v} ( D ) &= \Sigma \left( C^r \cap \mathscr{B}^e \right) - \Sigma \left( C^r \cap \mathscr{B}^o \right) \\
			&\quad + 2 \left( \sum_{s=1}^{i} \left( \Sigma \left( C^r \cap C^s \cap \mathscr{B}^o \right) - \Sigma \left( C^r \cap C^s \cap \mathscr{B}^e \right) \right) \right) \\
			\notag &= J_r ( D ) - J_{\mathscr{B}} ( D ) + 2 \left( \sum_{s=1}^{i} \left( \Sigma \left( C^r \cap C^s \cap \mathscr{B}^o \right) - \Sigma \left( C^r \cap C^s \cap \mathscr{B}^e \right) \right) \right).			
	\end{align}
	Denote by \( P(i) \) the right hand side of \Cref{3Eq:induction} for \( m = i \). Then
	\begin{equation*}
		\begin{aligned}
			P(i+1) - P(i) &= \left( \sum_{k=1}^{i} J_{r,k} ( D ) \right) - (i-1) J_r ( D ) - \left( \sum_{l=1}^i J_l (D) \right) + (i-1) J_{\mathscr{B}} (D) \\
			&= \left( \sum_{k=1}^{i}  J_{r,k} (D) - J_k ( D) \right) - (i-1) J_r (D) + (i-1) J_{\mathscr{B}} (D) \\
			&= \left( \sum_{k=1}^{i}  J_r (D) + J_k (D) + 2 \left( \Sigma \left( C^r \cap C^k \cap \mathscr{C}^o \right) \right) - 2 \left( \Sigma \left( C^r \cap C^k \cap \mathscr{C}^e \right) \right) \right. \\
			&\quad + J_{\mathscr{B}} (D) - J_k (D) \Bigg) - (i-1) J_r (D) + (i-1) J_{\mathscr{B}} (D) \\
			&= i J_r ( D ) - i J_{\mathscr{B}} (D) + 2\left( \sum_{k=1}^{i} \left( \Sigma \left( C^r \cap C^k \cap \mathscr{C}^o \right) \right) - \left( \Sigma \left( C^r \cap C^k \cap \mathscr{C}^e \right) \right) \right) \\
			&\quad - (i-1) J_r (D) + (i-1) J_{\mathscr{B}} (D) \\
			&= J_r (D) -J_{\mathscr{B}} (D) + 2 \left( \sum_{k=1}^{i} \left( \Sigma \left( C^r \cap C^k \cap \mathscr{C}^o \right) \right) - \left( \Sigma \left( C^r \cap C^k \cap \mathscr{C}^e \right) \right) \right). \\
		\end{aligned}
	\end{equation*}
	Comparing the final line above to that of \Cref{3Eq:difference} completes the proof.
\end{proof}

\subsection{Chequerboard colourable links}\label{3Subsec:cbable}

In the case of chequerboard colourable virtual links the \(2\)-colour parity, and hence the \(2\)-colour writhe, become easier to determine. To see this we require the notion of an \emph{abstract link diagram} \cite{Carter2000, Kamada2000}.

Let \( D \) be a diagram of a virtual link, as in \Cref{3Fig:ALD}, then
\begin{enumerate}
	\item about the classical crossings place a disc as shown in  \Cref{3Fig:ASDCrossing}
	\item about the virtual crossings place two discs as shown in \Cref{3Fig:ASDVirtual}
	\item join up these discs with collars about the arcs of the diagram.
\end{enumerate}
The result is a link diagram on a surface (which deformation retracts onto the underlying curve of the diagram). We denote this diagram and surface pair by \( F_D \), and refer to it as the \emph{abstract link diagram associated to \( D \)}. An example of an abstract link diagram is given in \Cref{3Fig:ALD}.

\begin{figure}
	\includegraphics[scale=1]{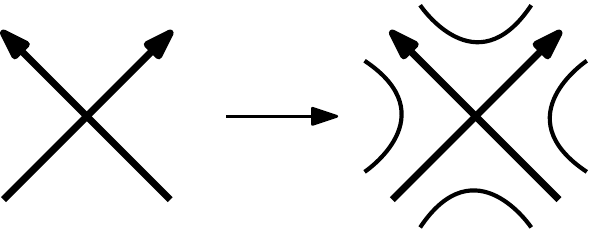}
	\caption{Component of the surface of an abstract link diagram about a classical crossing.}
	\label{3Fig:ASDCrossing}
\end{figure}
\begin{figure}
	\includegraphics[scale=1]{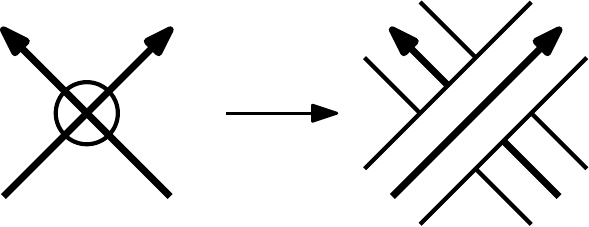}
	\caption{Component of the surface of an abstract link diagram about a virtual crossing.}
	\label{3Fig:ASDVirtual}
\end{figure}
\begin{figure}
	\includegraphics[scale=0.4]{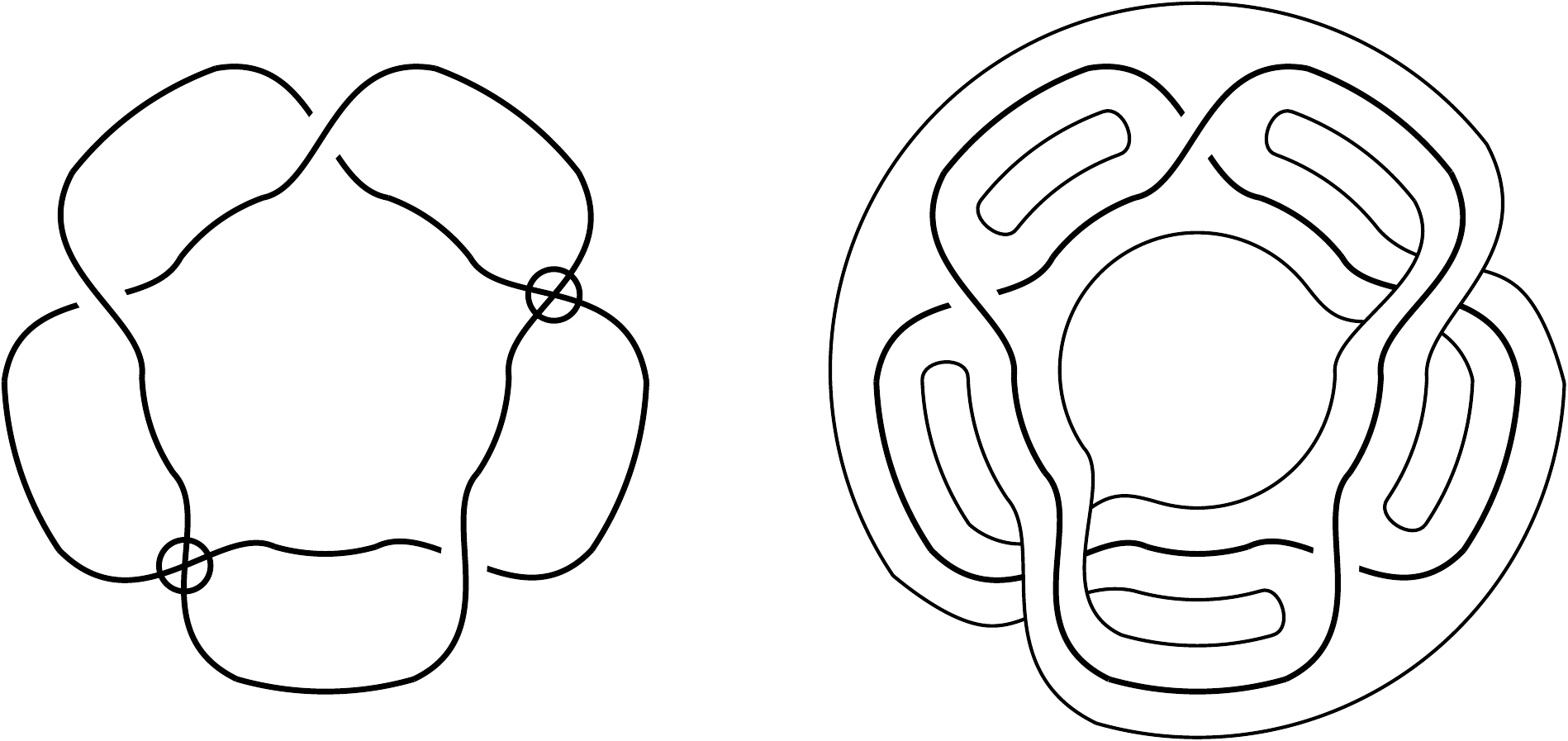}
	\caption{A virtual knot diagram and its associated abstract link diagram.}
	\label{3Fig:ALD}
\end{figure}

\begin{definition}
		\label{3Def:cbable}
		Let \( D \) be a diagram of a virtual link \( L \). Form the associated abstract link diagram \( F_D \) as described above. We say that \( F_D \) is \emph{chequerboard colourable} if a shading can be placed on the components of the complement of the underlying curves of the link diagram in \( F_D \) such that at each classical crossing we have the following (up to rotation)
		\begin{center}
			\includegraphics[scale=1]{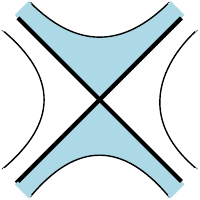}
		\end{center}
		A virtual link is \emph{chequerboard colourable} if it possesses a diagram whose associated abstract link diagram is chequerboard colourable.
\end{definition}

A chequerboard colourable virtual link is \(2\)-colourable (the converse is not true, however; counterexamples are given in \Cref{2Fig:knot21,3Fig:example1}).

\begin{theorem}
	\label{3Thm:cbwrithe}
	Let \( L \) be an oriented virtual link. The following are equivalent
	\begin{enumerate}[label=(\roman*)]
		\item \( L \) is chequerboard colourable
		\item \( L \) is \(2\)-colourable, and there exists a \(2\)-colouring of a diagram of it such that every classical crossing is even with respect to the parity associated to this \(2\)-colouring.
	\end{enumerate}
\end{theorem}

\begin{proof}
	(i)\(\Rightarrow\)(ii): Let \( D \) be a diagram of \( L \). We show that picking a chequerboard colouring of \( F_D \) distinguishes a \(2\)-colouring of \( D \), and that every classical crossing is even with respect to the parity associated to this \(2\)-colouring.
	
	First, we take an argument given by Bar-Natan and Morrison in the plane \cite{Bar-Natan2006}, and reproduce it on a surface. Form the associated abstract link diagram \( F_D \) as described above. Place a chequerboard colouring on \( F_D \), and a clockwise orientation on the shaded regions. To produce a \(2\)-colouring, compare the orientation of the arcs of \( S(D) \) (as inherited from \( D \)) to that induced by the orientation of the shaded region. If the orientations agree, colour the arc red, otherwise colour it green. That this yields a \(2\)-colouring is made clear in the following figures:
	\begin{center}
		\includegraphics[scale=0.75]{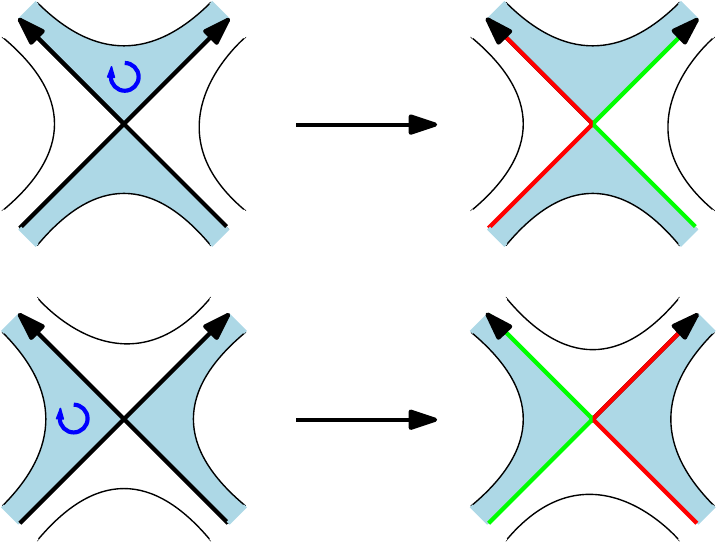}
	\end{center}
	It is also clear from these figures that \( \cp_{\mathscr{C}} ( c ) = 0 \) for all classical crossings \(c \) of \( D \), so that \( J_{\mathscr{C}} ( D ) = 0 \) and \( 0 \) appears in \( J^2 ( L ) \).

	(ii)\(\Rightarrow\)(i): Let \( D \) be a diagram of \( L \), and \( \mathscr{C} \) be a \(2\)-colouring of \( D \) such that every classical crossing is even with respect to \( \cp_{\mathscr{C}} \). Form \( F_D \); it is clear from the diagrams above that shading on the right of red arcs, and on the left or green arcs of \( D \) yields a chequerboard colouring of \( F_D \).
\end{proof}

It follows that if \( L \) is a chequerboard colourable oriented virtual link then \( 0 \) appears in \( J^2 ( L ) \).

\begin{corollary}
	\label{3Cor:classical}
	Let \( L \) be an oriented classical link. Then \( L \) is \(2\)-colourable, and there exists a \(2\)-colouring of a diagram \( D \) (of \( L \)) such that every classical crossing is even with respect to the parity associated to this \(2\)-colouring.
\end{corollary}

\begin{proof}
	Every classical link is chequerboard colourable: any chequerboard colouring of the plane will do.
\end{proof}

Using the distinguished \(2\)-colouring described in the proof of \Cref{3Thm:cbwrithe}, we observe that the \(2\)-colour writhe of a chequerboard colourable link may be determined from the pairwise linking numbers of its components.

\begin{definition}
	\label{3Def:linkingnumber} Let \( D = D_1 \cup D_2 \cup \cdots \cup D_n \) be a virtual link diagram with components \( D_i \), for \( i \in \lbrace 1, \ldots , n \rbrace \). Denote by \( lk ( D_i, D_j ) \) the sum of the signs of the mixed crossings between \( D_i \) and \( D_j \). Quantities such as \( lk ( D_i \cup D_j , D_k ) \) and \( lk ( D_i, D ) \) are defined similarly.
\end{definition}

Note that this definition of \( lk ( D_i, D_j ) \) does not contain the factor of \( \frac{1}{2} \) present in the classical definition.

\begin{proposition}
	\label{3Prop:cbwrithe2}
	Let \( D \) be a diagram of a chequerboard colourable oriented virtual link \( L \), taken with a chequerboard colouring of \( F_D \). Let \( \mathscr{C} \) be the \(2\)-colouring of \( D \) distinguished by this chequerboard colouring (as constructed in the proof of \Cref{3Thm:cbwrithe}). Let \( \mathcal{G} \) be the generating set with base colouring \( \mathscr{C} \) (described in \Cref{3Def:nicegenerators}). Then
	\begin{equation*}
		\begin{aligned}
			&J_i ( D ) = lk(D_i, D) \\
			&J_{i,j} ( D ) = lk ( D_i \cup D_j , D )
		\end{aligned}
	\end{equation*}
	where \( D_i \) is the \(i\)-th component of \( D \).
\end{proposition}

\begin{proof}
	As observed in the proof of \Cref{3Thm:cbwrithe}, every classical crossing of \( D \) is even with respect to the parity \( \cp_{\mathscr{C}} \), so that \( \mathscr{C}^{o} = \emptyset \) and \( C^i \cap \mathscr{C}^e = C^i \) (for \( C^i \), \( \mathscr{C}^e \) and \( \mathscr{C}^{o} \) as defined on \cpageref{3Eq:setdefinitions}) and \( J_{\mathscr{C}} ( D ) = 0 \). \Cref{3Eq:1dual,3Eq:2dual} become
		\begin{equation*}
			\begin{aligned}
				J_i ( D ) &= \sum_{c \in C^i } \text{sign}( c ) = lk ( D_i, D ) \\
				J_{i,j} ( D ) &= J_i + J_j - 2 \left( \sum_{c \in C^i \cap C^j } \text{sign}( c ) \right) \\
								&= lk ( D_i, D ) + lk ( D_j, D ) - 2lk ( D_i, D_j ) \\
								&= lk ( D_i \cup D_j , D )
			\end{aligned}
		\end{equation*}
	as required.
\end{proof}

Combining this with \Cref{3Thm:induction}, we see that the \(2\)-colour writhe of a chequerboard colourable virtual link may be determined from the pairwise linking numbers of its components. For non-chequerboard colourable virtual links, however, the \(2\)-colour writhe is a stronger invariant. An example of this is given in \Cref{3Fig:example1}: the virtual link depicted is of linking number \( 0 \), but has non-trivial \(2\)-colour writhe.

We conclude this section by demonstrating that the techniques of parity projection may be applied to links via the \( 2 \)-colour parity. Parity projection is a powerful construction that has been used to prove a number of results concerning virtual knots \cite{Manturov2010}, but its application to virtual links has been limited. This is due to the technical restrictions of parity theories for virtual links. In contrast, the \( 2 \)-colour parity lends itself naturally to parity projection.

Let \( D \) be an oriented virtual link diagram and \( \mathscr{C} \) be a \( 2 \)-colouring of it. Denote by \( \overline{D}^{\mathscr{C}} \) the diagram obtained from \( D \) by replacing all classical crossings of \( D \) which are odd with respect to \( \cp_{\mathscr{C}} \) with virtual crossings. The diagram \( \overline{D}^{\mathscr{C}} \) is known as the \emph{projection of \( D \) with respect to \( \cp_{\mathscr{C}} \) }.

Parity projection interacts well with the virtual Reidemeister moves; we outline necessary results here, and refer the interested reader to \cite[Section \(3\)]{Manturov2010}. If \( D' \) is related to \( D \) via a sequence of virtual Reidemeister moves then the \(2\)-colouring \( \mathscr{C} \) induces a \(2\)-colouring of \(D'\), denoted \( \mathscr{C}' \). Then the projection \( \overline{D'}^{\mathscr{C}'} \) is related to \( \overline{D}^{\mathscr{C}} \) via a sequence of virtual Reidemeister moves.

In what follows a \emph{minimal crossing diagram} of a virtual link \( L \) is a diagram with the minimum number of classical crossings, with the minimum taken across all diagrams of \( L \). We demonstrate the utility of parity projection for links by using it to demonstrate that a minimal crossing diagram of a chequerboard colourable virtual link is itself chequerboard colourable. The corresponding result for virtual knots may be proved using the Gaussian parity, but the extension to links requires the \(2 \)-colour parity. The author thanks Hans Boden and Homayun Karimi for sharing the following argument with him.

\begin{proposition}\label{3Prop:cbminimal}
Let \( L \) be a chequerboard colourable oriented virtual link. A minimal crossing diagram of \( L \) is chequerboard colourable.
\end{proposition}

\begin{proof}
Let \( D \) be a minimal crossing diagram of \( L \), and \( D ' \) a chequerboard colourable diagram of \( L \). As \( D ' \) is chequerboard colourable, \Cref{3Thm:cbwrithe} guarantees that there exists a \(2\)-colouring of it, \( \mathscr{C}' \), such that every classical crossing of \( D'\) is even with respect to \( \cp_{\mathscr{C}'} \). It follows that \( \overline{D'}^{\mathscr{C}'} = D' \).

The diagrams \( D \) and \( D' \) both represent the \( L \), so they are related via a sequence of virtual Reidemeister moves. Then the \( 2 \)-colouring \( \mathscr{C}' \) of \( D '\) induces a \( 2 \)-colouring of \( D \), denoted \( \mathscr{C} \). Further, the projections \( \overline{D}^{\mathscr{C}} \) and \( \overline{D'}^{\mathscr{C}'} \) are related via a sequence of virtual Reidemeister moves. But \( \overline{D'}^{\mathscr{C}'} = D' \), so that \( \overline{D}^{\mathscr{C}} \) is also a diagram of \( L \). If a classical crossing of \( D \) was converted to a virtual crossing when passing to \( \overline{D}^{\mathscr{C}} \), then \( \overline{D}^{\mathscr{C}} \) would be a diagram of \( L \) with fewer classical crossings than \( D \). But \( D \) is a minimal crossing diagram, and it follows that no classical crossings are converted to virtual crossings. Equivalently, every classical crossing of \( D \) is even with respect to \( \cp_{\mathscr{C}} \), and by \Cref{3Thm:cbwrithe} we conclude that \( D \) is a chequerboard colourable diagram.
\end{proof}

\section{Concordance invariance}\label{4Sec:concordance}

In this section we prove that the \(2\)-colour writhe is invariant under virtual link concordance. There are at least two reasons to suspect that this should be the case. First, as demonstrated in \Cref{3Prop:oddwrithe}, the \(2\)-colour writhe restricts to the odd writhe on virtual knots, which is itself invariant under virtual knot concordance \cite{Boden2018,Rushworth2017}. If the \( 2 \)-colour writhe is to be a satisfactory extension of the odd writhe to virtual links, it should replicate this behaviour.

Second, the pairwise linking numbers of a virtual link are elementary concordance invariants. In \Cref{3Prop:cbwrithe2} it is shown that the \(2\)-colour writhe of a chequerboard colourable virtual link may be determined from its pairwise linking numbers, so that the \(2\)-colour writhe is a concordance invariant on such virtual links.

To prove the concordance invariance of the \(2\)-colour writhe of an arbitrary \(2\)-colourable oriented virtual link we use \emph{doubled Lee homology}. In \Cref{4Subsec:dkh} we show that the \(2\)-colour writhe of a virtual link appears as the homological degree support of its doubled Lee homology. The concordance invariance of the \(2\)-colour writhe follows from this result and the concordance invariance of doubled Lee homology itself, as we describe in \Cref{4Subsec:invariance}. In \Cref{4Subsec:applications} we obtain applications of this concordance invariance.

\subsection{The \(2\)-colour writhe and doubled Lee homology}\label{4Subsec:dkh}
In this section we show that the \(2\)-colour writhe of a virtual link is equivalent to the homological degree support of its doubled Lee homology. We keep our description of doubled Lee homology to a minimum, referring the interested reader to \cite{Rushworth2017} for full details.

The \emph{doubled Lee homology} of an oriented virtual link \( L \), denoted \( \dkh ' ( L ) \), is a bigraded finitely-generated Abelian group. The two gradings are the \emph{homological grading}, denoted \( i \), and the \emph{quantum grading}, denoted \( j \) (both are \( \Z \)-gradings). We focus on the homological grading.

\begin{definition}
	\label{4Def:smoothing}
	Let \( D \) be a virtual link diagram. A classical crossing of \( D \) may be \emph{resolved} in one of two ways:
	\begin{center}
		\includegraphics[scale=0.75]{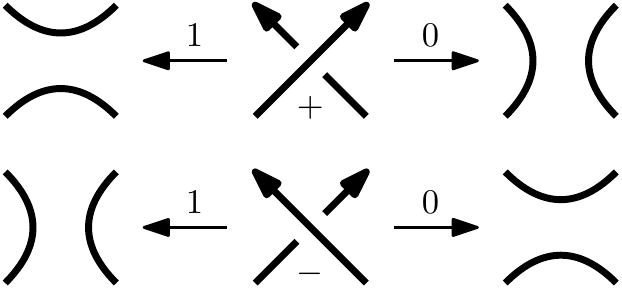}
	\end{center}
	We refer to each resolution as either the \(0\)- or the \(1\)-resolution, depending on the sign of the crossing. Arbitrarily resolving every classical crossing of \( D \) (leaving virtual crossings unchanged) yields a collection of immersed circles in the plane, known as a \emph{smoothing of \(D\)}.
\end{definition}

\begin{definition}\label{4Def:height}
	Let \( D \) be an oriented virtual link diagram with \( n_- \) negative classical crossings. Given \( \mathscr{S} \) a smoothing of \( D \) with exactly \( m \) \(1\)-resolutions, define the \emph{height} of \( \mathscr{S} \) as
		\begin{equation*}
		| \mathscr{S} | \coloneqq m - n_-.\qedhere
		\end{equation*}
\end{definition}

\begin{definition}
	\label{4Def:acs}
	A smoothing of a virtual link diagram \( D \) is \emph{alternately coloured} if its cycles are coloured exactly one of two colours in such a way that in a neighbourhood of each classical crossing the two incident arcs have different colours. A smoothing that can be coloured in such a way is known as an \emph{alternately colourable}.
\end{definition}
Examples of alternately coloured smoothings are given on the left of \Cref{4Fig:borrocolour}.
	
Let \( D \) be a diagram of an oriented virtual link \( L \). In \cite[Section \(3\)]{Rushworth2017} it is shown that to each alternately coloured smoothing of \( D \) one may associate two elements \( \sg^u, \sg^l \in \dkh ' ( L ) \), known as \emph{alternately coloured generators}. This name is justified by the fact that the set of all such objects generates \( \dkh ' ( L ) \) \cite[Theorem \( 3.5 \)]{Rushworth2017}. Define the homological grading on alternately coloured generators as
\begin{equation}\label{4Eq:hdegree}
	i ( \sg^{u/l} ) = | \mathscr{S} |
\end{equation}
where \( \sg^{u/l} \) are associated to the alternately coloured smoothing \( \mathscr{S} \).

The alternately coloured smoothings of \( D \) are directly related to its \( 2 \)-colourings.
\begin{proposition}\label{4Prop:2colourACG}
	Let \( D \) be a diagram of a \(2\)-colourable oriented virtual link \( L \). There is a bijection between the set of \(2\)-colourings of \( D \) and that of alternately coloured smoothings of \( D \).
\end{proposition}

\begin{proof}
	Given a \(2\)-colouring one may produce an alternately coloured smoothing, and vice versa, using the following rule:
		\begin{center}
				\includegraphics[scale=0.75]{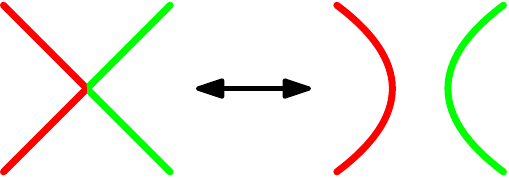}
		\end{center}
	That this is a bijection is clear.
\end{proof}
Examples of this bijection are given in \Cref{4Fig:borrocolour}.

\begin{figure}
	\includegraphics[scale=0.5]{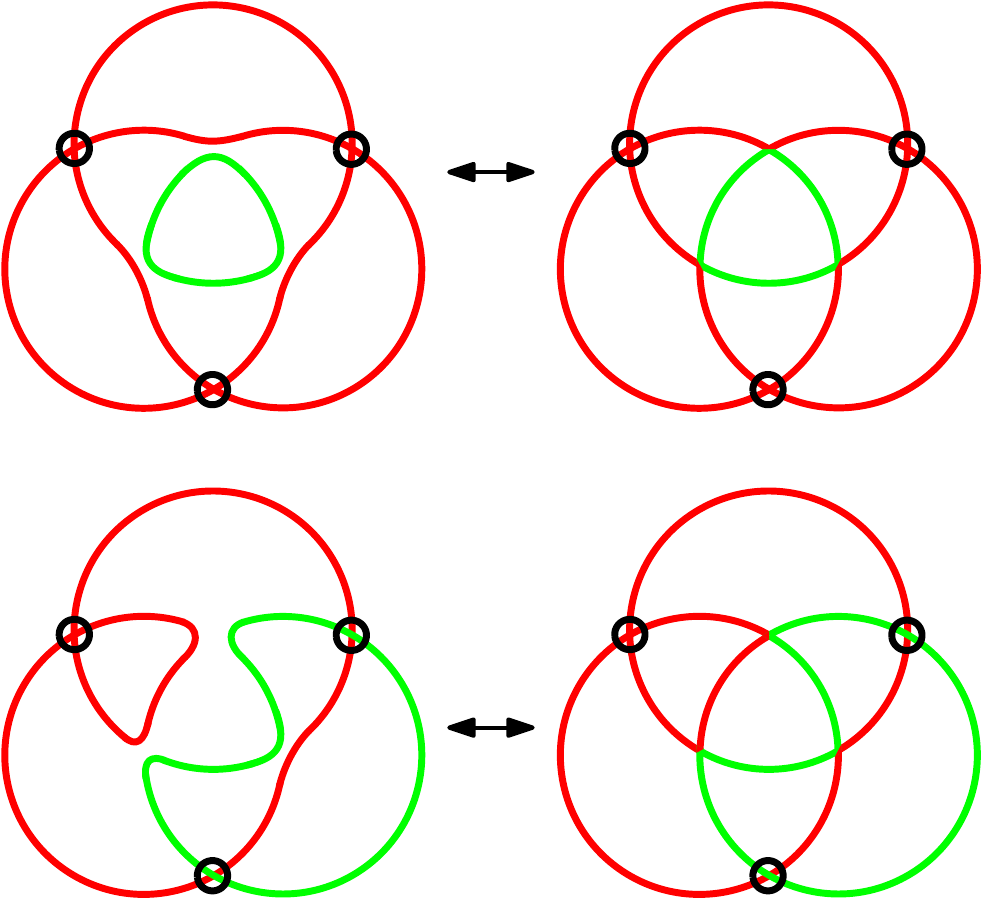}
	\caption{The bijection between alternately coloured smoothings and \(2\)-colourings.}
	\label{4Fig:borrocolour}
\end{figure}

The following is a generalization of \cite[Proposition \( 4.11 \)]{Rushworth2017}.

\begin{lemma}
	\label{4Lemma:writhedegree}
	Let \( D \) be a diagram of a \(2\)-colourable oriented virtual link \( L \). Let \( \mathscr{C} \) be a \( 2 \)-colouring of \( D \) and \( \mathscr{S} \) the alternately coloured smoothing of \( D \) associated to \( \mathscr{C} \). If \( \sg \) is either of the alternately coloured generators associated to \( \mathscr{S} \) then
		\begin{equation*}
			J_{\mathscr{C}} ( D ) = i ( \sg ).
		\end{equation*}
\end{lemma}

\begin{proof}
	Let \( D \) have \( n^e_+ \) (\( n^e_- \)) even positive (negative) classical crossings, and \( n^o_+ \) (\( n^o_- \)) odd positive (negative) classical crossings, with respect to \( \cp_{\mathscr{C}} \). Then \( n_+ = n^e_+ + n^o_+ \) (\( n_- = n^e_- + n^o_- \)) is the total number of positive (negative) classical crossings.
	
	Given a \(2\)-colouring \( \mathscr{C} \), and its associated alternately coloured smoothing \( \mathscr{S} \), one of the following configurations must occur:
	\begin{center}
		\includegraphics[scale=0.65]{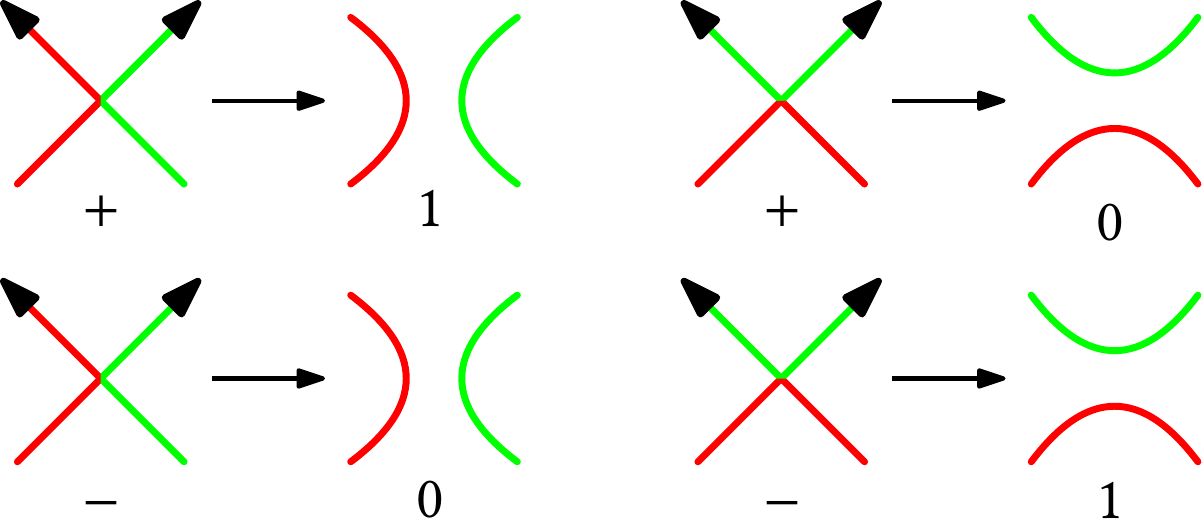}
	\end{center}
	where \( +, - \) denotes the sign of the crossing and \( 0,1 \) its resolution in \( \mathscr{S} \). That is, at a classical crossing of \( D \) exactly one of the following cases holds
	\begin{center}
		\begin{tabular}{c | c | c }
			sign  & parity & reso \\ \hline
			\(+\) & odd & \(1\) \\
			\(+\) & even & \( 0 \) \\
			\(-\) & odd & \(0\) \\
			\(-\) & even & \(1\)
		\end{tabular}
	\end{center}
	Using the table above, and recalling \Cref{4Eq:hdegree}, we have
	\begin{equation*}
		\begin{aligned}
			J_{\mathscr{C}} ( D ) &= n^o_+ - n^o_- \\
			&= n^o_+ - n^o_- + n^e_- - n^e_- \\
			&= n^o_+ + n^e_- - n_- \\
			&= i ( \sg )
		\end{aligned}
	\end{equation*}
	as required.
\end{proof}

We are now able to conclude this section by demonstrating that the \( 2 \)-colour writhe is equivalent to the homological support of doubled Lee homology.
\begin{theorem}\label{4Thm:writhedegree}
	Let \( L \) be a \(2\)-colourable oriented virtual link. Denote by \( \lbrace J^2 ( L ) \rbrace \) the set of entries of \( J^2 ( L ) \). Then \( \lbrace J^2 ( L ) \rbrace \) is the homological degree support of \( \dkh ' ( L ) \).
\end{theorem}

\begin{proof}
	Let \( D \) be a diagram of \( L \). Pick a generating set of \(2\)-colourings and compute \( J^2 ( L ) \).
	
	Let \( k \) be an element of the homological degree support of \( \dkh ' ( L ) \) i.e.\ a homological degree in which \( \dkh ' ( L ) \) is non-trivial. Then there exists an alternately coloured generator, \( \sg \), such that \( i ( \sg ) = k \) (we have suppressed the superscript \( u/l \) as \( i ( \sg^u ) = i ( \sg^l ) \), c.f.\ \Cref{4Eq:hdegree}). Let \( \mathscr{S} \) be the alternately coloured smoothing of \( D \) associated to \( \sg \). By \Cref{4Prop:2colourACG} there is a \(2\)-colouring of \( D \) associated to \( \mathscr{S} \), denoted \( \mathscr{C} \). If \( \mathscr{C} \) is in the generating set of \( 2 \)-colourings then
	\begin{equation*}
		k = i ( \sg ) = J_{\mathscr{C}} ( D ) \in \lbrace J^2 ( L ) \rbrace
	\end{equation*}
	by \Cref{4Lemma:writhedegree}. If \( \mathscr{C} \) is not in the generating set, then \( \overline{\mathscr{C}} \) is, and
	\begin{equation*}
		k = i ( \sg ) = J_{\mathscr{C}} ( D ) = J_{\overline{\mathscr{C}}} ( D ) \in \lbrace J^2 ( L ) \rbrace.
	\end{equation*}
	
	Conversely, let \( \mathscr{C} \) be an element of the generating set. By \Cref{4Prop:2colourACG} there is an associated alternately coloured smoothing of \( D \), \( \mathscr{S} \). Let \( \sg \) be either of the alternately coloured generators of \( \dkh ' ( L ) \) associated to \( \mathscr{S} \), then by \Cref{4Lemma:writhedegree}
	\begin{equation*}
		J_{\mathscr{C}} ( D ) = i ( \sg )
	\end{equation*}
	so that \( J_{\mathscr{C}} ( D ) \) is an element of the homological degree support of \( \dkh ' ( L ) \).
\end{proof}

\subsection{Invariance}\label{4Subsec:invariance}
In this section we use \Cref{4Thm:writhedegree} and the concordance invariance of doubled Lee homology to show that the \(2\)-colour writhe is itself a concordance invariant. First, we recall the topological interpretation of virtual links and cobordisms between them.

\begin{definition}
	\label{4Def:vlink}
	A \emph{virtual link} is an equivalence class of embeddings \( \bigsqcup S^1 \hookrightarrow \Sigma_g \times I \) up to self-diffeomorphism of \( \Sigma_g \times I \), and handle (de)stabilisations of \( \Sigma_g \) such that the product of the attaching sphere with the \( I \) summand of \( \Sigma_{g} \times I \) is disjoint from the image of the embedding.
	
	A \emph{representative \( D \) of a virtual link} is a particular embedding \( D : \bigsqcup S^1 \hookrightarrow \Sigma_g \times I \). We abbreviate the notation to write \( D \hookrightarrow \Sigma_g \times I \).
\end{definition}

\begin{definition}[\cite{TuraevCob}]
	\label{4Def:cobordism}
	Let \( D \hookrightarrow \Sigma_g \times I \) and \( D' \hookrightarrow \Sigma_{g'} \times I \) be representatives of virtual links \( L \) and \( L ' \). We say that \( L \) and \( L ' \) are \emph{cobordant} if there exist a compact oriented surface \( S \) and an oriented \( 3 \)-manifold \( M \), such that \( S \hookrightarrow M \times I \), \( \partial S = D \sqcup D' \), and \( \partial M = \Sigma_g \sqcup \Sigma_{g'} \). We refer to \( S \) as a \emph{cobordism between \( L \) and \( L' \)}.
\end{definition}

\begin{definition}
	\label{4Def:concordance}
	Let \( S \hookrightarrow M \times I \) be a cobordism between links \( L \) and \( L' \), with \( | L | = | L' | \)\footnote{recall that \( | L | \) denotes the number of components of \( L \)}. We say that \( L \) and \( L' \) are \emph{concordant} if there exists a cobordism \( S \) between them that is a disjoint union of \( | L | \)  annuli, such that each annulus has a boundary component in \( L \) and another in \( L' \). We refer to such an \( S \) as a \emph{concordance between \( L \) and \( L' \)}.
\end{definition}

There exist algorithms for moving between this topological interpretation and the diagrammatic interpretation used in the previous sections of this paper; for example, see \cite[Section \( 1.1\)]{Boden2018}.

Importantly for our purposes, \(2\)-colourability of virtual links is preserved throughout concordances.
\begin{proposition}
	\label{4Prop:2colourconcordance}
	Let \( L \) and \( L' \) be concordant oriented virtual links. Then \( L \) is \( 2 \)-colourable if and only if \( L ' \) is \(2\)-colourable.
\end{proposition}

\Cref{4Prop:2colourconcordance} can be proved using the following fact.

\begin{lemma}\label{4Lem:linkingnumber}
Let \( L \) and \( L ' \) be concordant oriented virtual links, with \( L = L_1 \cup L_2 \cup \cdots \cup L_n \) and \( L' = L^{\prime}_1 \cup L^{\prime}_2 \cup \cdots \cup L^{\prime}_n \). As \( L \) and \( L ' \) are concordant there is a bijection between the components of \( L \) and those of \( L ' \); without loss of generality let \( L_i \) be associated to \( L^{\prime}_i \). Then
\begin{equation*}
lk ( L_i , L_j ) = lk ( L^{\prime}_i, L^{\prime}_j )
\end{equation*}
for all \( 1 \leq i, j \leq n \).
\end{lemma}

\begin{proof}
	The proof of this fact in the case of classical links is elementary, and it may be adapted to the virtual case as follows. By abuse of notation let \( D = D_1 \cup D_2 \cup \cdots \cup D_n \hookrightarrow \Sigma_g \times I \) be a representative of \( L \). Define
	\begin{equation*}
	B_m = \faktor{\Sigma_g \times I}{\Sigma_g \times \lbrace m \rbrace}, ~ m \in \lbrace 0, 1 \rbrace
	\end{equation*}
	Consider the inclusion \( D \hookrightarrow B_m \); as \( B_m \) is a homology disc Alexander duality induces the isomorphism
	\begin{equation*}
	\phi_{m,j} : H_{1} \left( B_m \setminus D_j \right) \rightarrow \Z
	\end{equation*}
	for all \( j \in \lbrace 1, \ldots, n \rbrace \).
	
	Next, project \( D \) to \( \Sigma_g \). This yields a link diagram on \( \Sigma_g \) (without virtual crossings). By abuse of notation we denote this diagram by \( D = D_1 \cup D_2 \cup \cdots \cup D_n \) also. Define the quantities
	\begin{equation*}
	\begin{aligned}
	lk_0 \left( D_i , D_j \right) &= \sum_{D_i ~\text{under}~ D_j ~ \text{at}~ c} \text{sign} ( c ) \\
	lk_1 \left( D_i , D_j \right) &= \sum_{D_i ~\text{over}~ D_j ~ \text{at}~ c} \text{sign} ( c )
	\end{aligned}
	\end{equation*}
	It follows that
	\begin{equation}\label{4Eq:ident}
	lk_m \left( D_i , D_j \right) = \phi_{1-m,j} \left( [ D_i ] \right)
	\end{equation}
	for \( m \in \lbrace 0 , 1 \rbrace \). Notice that \( lk ( D_i, D_j ) = lk_0 \left( D_i , D_j \right) + lk_1 \left( D_i , D_j \right) \). Thus the linking number of the components \( D_i \) and \( D_j \) is a function of \( [D_i] \in H_1 \left( B_m \setminus D_j \right) \), in direct analogy to the classical case. With this in place, we may repeat the classical proof: a concordance between links may be used to exhibit homotopies of the complements of link components. These homotopies preserve homology classes, so that by \Cref{4Eq:ident} the linking numbers are preserved under concordance.
\end{proof}

\begin{proof}[Proof of \Cref{4Prop:2colourconcordance}]
	Let \( D = D_1 \cup D_2 \cup \cdots \cup D_n \) be a diagram of \( L \), and \( G ( D ) \) its simple Gauss diagram as defined in \Cref{2Def:gauss}. Recall that a circle of \( G ( D ) \) is known as degenerate if it contains an odd number of chord endpoints (see \Cref{2Def:degen}). The components of \( D \) are in bijection with the circles of \( G ( D ) \); in what follows we shall refer to a component \( D_i \) as degenerate if the associated circle of \( G ( D ) \) is degenerate. It then follows from \Cref{2Thm:bijection} that \( L \) is \( 2 \)-colourable if and only \( D_i \) is not degenerate, for all \( 1 \leq i \leq n \). 

	It is readily verified that \( D_i \) is degenerate if and only if \( lk ( D_i, D \setminus D_i ) \) is odd. Noticing that
	\begin{equation*}
	lk ( D_i, D \setminus D_i ) = \sum_{j=1, j \neq i}^n lk ( D_i, D_j )
	\end{equation*}
	and employing \Cref{4Lem:linkingnumber}, we see that if \( L \) is concordant to \( L ' \) then \( L \) has a degenerate component if and only if \( L ' \) has. Thus \( L \) is \(2\)-colourable if and only if \( L ' \) is.
\end{proof}

A powerful feature of doubled Lee homology is its \emph{functoriality}: to a cobordism \( S \) between virtual links \( L \) and \( L' \), the doubled Lee package assigns a map \( \phi_S : \dkh ' ( L ) \rightarrow \dkh ' ( L ') \). If \( S \) is a concordance then \( \phi_S \) is an isomorphism.

\begin{proposition}
	\label{4Prop:isomorphism}
	Let \( S \) be a concordance between virtual links \( L \) and \( L' \). The map \( \phi_S : \dkh ' ( L ) \rightarrow \dkh ' ( L ') \) is a bigraded isomorphism of homological degree \(0\).
\end{proposition}

\begin{proof}
	The case in which one of \( L \) and \( L' \) is a knot is given in \cite[Theorem \( 3.21 \)]{Rushworth2017}, and the argument readily extends to the case in which both \( L \) and \( L' \) are links. Here we outline the proof, directing the reader to \cite[Theorem \( 3.21 \)]{Rushworth2017} for full details.

	Recall from \Cref{4Subsec:dkh} that the bigraded group \( \dkh ' ( L ) \) has generators that are in bijection with the \(2\)-colourings of \( L \). Further, the homological degree of each generator is equal to the contribution to the \(2\)-colour write of \( L \) of the associated \(2\)-colouring. Thus, given a concordance \( S \) from \( L \) to \( L '\), one can study the map \( \phi_S \) by analysing the effects of \( S \) on the set of \(2\)-colourings of \( L \). In \cite[Section \( 3.2 \)]{Rushworth2017} it is demonstrated that passing through a concordance cannot create degenerate link components; thus every link that appears in a concordance from \( L \) to \( L ' \) is \(2\)-colourable. It follows that the map \( \phi_S \) is non-zero. By considering the reverse concordance - that is, the concordance from \( L ' \) to \( L \) obtained by traversing \( S \) in reverse - one can prove that \( \phi_S \) is an isomorphism. The map \( \phi_S \) is of homological degree \( 0 \) by construction, as described in \cite[Section \( 3.2 \)]{Rushworth2017}.
\end{proof}

\begin{corollary}
	\label{4Cor:writheinvariance}
	Let \( L \) and \( L' \) be concordant \(2\)-colourable oriented virtual links. Then \( J^2 ( L ) = J^2 ( L' ) \).
\end{corollary}

\begin{proof}
	This is a direct result of \Cref{4Thm:writhedegree} and \Cref{4Prop:isomorphism}. We suffice ourselves by observing that not only are the entries of \( J^2 ( L ) \) concordance invariants, but their multiplicities are also, owing to the constructive nature of the proof of \Cref{4Thm:writhedegree}.
\end{proof}

\subsection{Applications}\label{4Subsec:applications}
The concordance invariance of the \(2\)-colour writhe established above yields a number of applications. First, we are able to obstruct the sliceness of a \(2\)-colourable oriented virtual link.

\begin{definition}
	\label{4Def:strongslice} A virtual link \( L \) is \emph{slice} if it is concordant to the unlink of \( | L | \) components.
\end{definition}

The following is a direct consequence of \Cref{4Cor:writheinvariance}.

\begin{corollary}\label{4Cor:strongslice}
	Let \( L \) be a \(2\)-colourable oriented virtual link. If \( J^2 ( L ) \) has a non-zero entry then \( L \) is not slice.
\end{corollary}

\Cref{3Fig:example1} provides an example of an oriented virtual link for which the linking number does not obstruct sliceness, but the \(2\)-colour writhe does.

Combining \Cref{4Cor:writheinvariance} with \Cref{3Prop:amphichiral} and \Cref{3Thm:cbwrithe} we can obstruct (\( \pm \))-amphichirality and chequerboard colourability within a concordance class.

\begin{corollary}
	\label{4Cor:amphiconc}
	Let \( L \) be a \(2\)-colourable oriented virtual link such that \( J^2 ( L ) \neq -J^2 ( L ) \). Then \( L \) is not concordant to a (\( \pm \))-amphichiral virtual link.
\end{corollary}

\begin{corollary}
	\label{4Cor:cbconc}
	Let \( L \) be a \(2\)-colourable oriented virtual link such that \( 0 \) does not appear in \( J^2 ( L ) \). Then \( L \) is not concordant to a chequerboard colourable virtual link.
\end{corollary}

\Cref{3Fig:example1} provides an example of an oriented virtual link that is obstructed from being concordant to a chequerboard colourable link by the \(2\)-colour writhe, and is therefore not concordant to a classical link.

\section{Comparison with other theories}\label{5Sec:comparison}

In this section we compare the \(2\)-colour parity to other parity theories for virtual links. In \Cref{5Subsec:naive,5Subsec:IP} we show that \(2\)-colour parity yields a strictly stronger invariant (on the set of \(2\)-colourable virtual links) than both the na\"\i ve parity and the IP parity. In \Cref{5Subsec:other} we show that the \(2\)-colour parity descends to a parity on free links, and compare it to other related topics.

\subsection{The na\"\i ve parity}\label{5Subsec:naive}
One can define a parity of virtual links by simply declaring that all self-crossings are even, and all mixed crossings odd. We refer to this parity as the \emph{na\"\i ve parity}.

\begin{definition}\label{5Def:naive}
	Let \( D \) be a virtual link diagram. Declare every classical self-crossing as \emph{N-even}, and every classical mixed crossing as \emph{N-odd}.
\end{definition}

It is easy to see that this declaration satisfies the axioms of \Cref{2Def:parityaxioms}, and yields a parity.

\begin{definition}\label{5Def:nwrithe}
	Let \( D \) be a diagram of an oriented virtual link \( L \). Define the \emph{na\"\i ve writhe} of \( D \), denoted \( N ( D ) \), as
	\begin{equation*}
		N ( D ) \coloneqq \sum_{c ~\text{N-odd}} \text{sign} ( c ).
	\end{equation*}
	That is, it is the sum of the N-odd crossings. As the na\"\i ve parity satisfies the axioms of \Cref{2Def:parityaxioms}, this quantity is an invariant of \( L \), and we define the na\"\i ve writhe of \( L \) to be \( N ( L ) \coloneqq N ( D ) \).
\end{definition}

Let \( M ( D ) \) denote the set of mixed crossings of \( D \). As a classical crossing of \( D \) is N-odd if and only if it is a mixed crossing, we observe that
\begin{equation*}
	N ( D ) = \sum_{c \in M(D) } \text{sign} ( c ).
\end{equation*}
Combining this with
\begin{equation*}
	\sum_{c \in C^i \cap C^j } \text{sign} ( c ) = lk ( D_i , D_j )
\end{equation*}
and
\begin{equation*}
	M ( D ) = \bigcup_{i} \left( \bigcup_{j>i} C^i \cap C^j  \right)
\end{equation*}
(where \( C^i \) is as defined on \cpageref{3Eq:setdefinitions}), we obtain
\begin{equation}\label{5Eq:nwrithe1}
	N ( D ) = \sum_{c \in M(D) } \text{sign} ( c ) = \sum_{i} \left( \sum_{j>i} lk ( D_i , D_j ) \right).
\end{equation}
That is, the na\"\i ve writhe is the sum of the pairwise linking numbers. It is demonstrated in \Cref{3Subsec:cbable} that the \(2\)-colour writhe is strictly stronger than the pairwise linking numbers, motivating the following theorem.

\begin{theorem}\label{5Thm:2colournwrithe}
	The pair \( \left( J^2 ( L ), J^2_S ( L ) \right) \) forms a strictly stronger invariant than \( N ( L ) \) on the set of \(2\)-colourable oriented virtual links (for \( J^2_S ( L ) \) given in \Cref{3Def:2colourselfwrithe}).
\end{theorem}

\begin{proof}
	We show that if \(L\) is a \(2\)-colourable oriented virtual link such that \( \left( J^2 ( L ), J^2_S ( L ) \right) \) is trivial, then \( N ( L ) \) is trivial also, and that there exist virtual links detected by \( \left( J^2 ( L ), J^2_S ( L ) \right) \) but not \( N ( L ) \).
	
	First, let \( L \) be a virtual link such that \( \left( J^2 ( L ), J^2_S ( L ) \right) \) is trivial. That is, given a diagram \( D = D_1 \cup D_2 \cup \cdots \cup D_n \), then \( J^2_S ( D ) = 0 \) and \( J_{\mathscr{C}} ( D ) = 0 \) for all \(2\)-colourings \( \mathscr{C} \) of \( D \). Fixing an arbitrary base colouring \( \mathscr{C} \), and using the notation of \Cref{3Subsec:complexity}, we obtain
	\begin{equation*}
		\begin{aligned}
			J_{i,j} ( D ) &= J_i (D) + J_j (D) + 2 \left( \Sigma \left( C^i \cap C^j \cap \mathscr{C}^o \right) \right) - 2 \left( \Sigma \left( C^i \cap C^j \cap \mathscr{C}^e \right) \right) - J_{\mathscr{C}} (D) \\
			&= 2 \left( \Sigma \left( C^i \cap C^j \cap \mathscr{C}^o \right) \right) - 2 \left( \Sigma \left( C^i \cap C^j \cap \mathscr{C}^e \right) \right) \\
			&= 0
		\end{aligned}
	\end{equation*}
	as \( J^2 ( L ) \) is trivial. Therefore
	\begin{equation*}
		\Sigma \left( C^i \cap C^j \cap \mathscr{C}^o \right) = \Sigma \left( C^i \cap C^j \cap \mathscr{C}^e \right)
	\end{equation*}
	and
	\begin{equation}\label{5Eq:linkingnumber}
		\begin{aligned}
			lk ( D_i, D_j ) &= \Sigma \left( C^i \cap C^j \cap \mathscr{C}^o \right) + \Sigma \left( C^i \cap C^j \cap \mathscr{C}^e \right) \\
			&= 2 \left( \Sigma \left( C^i \cap C^j \cap \mathscr{C}^o \right) \right).
		\end{aligned}
	\end{equation}
	Next, consider the following expression for \( J_{\mathscr{C}} \)
	\begin{equation*}
		\begin{aligned}
			J_{\mathscr{C}} &= \sum_{\substack{c ~\text{mixed crossing} \\ \cp_{\mathscr{C}} ( c ) = 1 }} \text{sign} ( c ) + J^2_{S} ( D ) \\
			&= 0
		\end{aligned}
	\end{equation*}
	from which we observe
	\begin{equation*}
		\sum_{\substack{c ~\text{mixed crossing} \\ \cp_{\mathscr{C}} ( c ) = 1 }} \text{sign} ( c ) = - J^2_{S} ( D ) = 0.
	\end{equation*}
	By definition
	\begin{equation*}
		\sum_{\substack{c ~\text{mixed crossing} \\ \cp_{\mathscr{C}} ( c ) = 1 }} \text{sign} ( c ) =  \sum_i \left(\sum_{j > i} \left( \Sigma \left( C^i \cap C^j \cap \mathscr{C}^o \right) \right) \right)
	\end{equation*}
	so that
	\begin{equation*}
		\begin{aligned}
			0 &=  \sum_i \left(\sum_{j > i} \left( \Sigma \left( C^i \cap C^j \cap \mathscr{C}^o \right) \right) \right) \\
			&= \frac{1}{2} \sum_i \left( \sum_{ j > i } lk ( D_i , D_j ) \right)\\
			&=\frac{1}{2} N ( L )
		\end{aligned}
	\end{equation*}
	by \Cref{5Eq:nwrithe1,,5Eq:linkingnumber}, and we may conclude that \( N (  L ) = 0 \).
	
	A virtual link that is detected by the \(2\)-colour writhe but not by the na\"\i ve writhe is given in \Cref{5Fig:cbdetected}.
\end{proof}

\begin{figure}
	\includegraphics[scale=0.65]{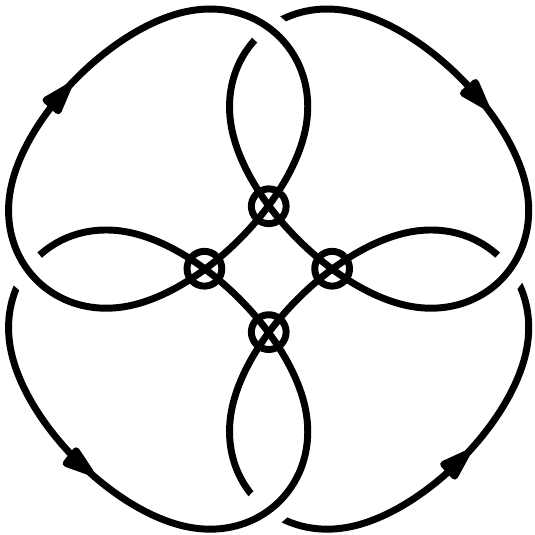}
	\caption{A virtual link detected by the \( 2 \)-colour writhe, but not detected by the na\"\i ve writhe or the IP writhe.}
	\label{5Fig:cbdetected}
\end{figure}

\subsection{The Im-Park parity}\label{5Subsec:IP}
It is possible for the pairwise linking numbers of a virtual link to be odd, unlike those of a classical link. In this section we shall consider the class of virtual links with even pairwise linking numbers. The set of such links is a subset of \(2\)-colourable links.

\begin{proposition}\label{5Prop:evensubset}
Let \( L \) be an oriented virtual link with even pairwise linking numbers. Then \( L \) is \(2\)-colourable.
\end{proposition}

\begin{proof}
Let \( D \) be a diagram of \( L \). Recall that \( L \) is \(2\)-colourable if and only if \( G ( D ) \) has no degenerate circles (see \Cref{2Thm:bijection}). As \( L \) has even pairwise linking numbers, a circle of \( G ( D ) \) must have an even number of chord endpoints. It follows that \( G ( D ) \) possesses no degenerate circles and \( L \) is \(2\)-colourable.
\end{proof}
However, there exist \(2\)-colourable virtual links which do not have even pairwise linking numbers; an example is given in \Cref{2Fig:shadowandgauss}.

Im and Park defined a parity on the restricted class of virtual links with even pairwise linking numbers \cite{ImPark2013}. We refer to this parity as the \emph{Im-Park parity} (IP parity). We begin by summarising its definition and that of the associated writhe invariant.

\begin{definition}[IP parity \cite{ImPark2013}]\label{5Def:ipparity}
	Let \( D = D_1 \cup D_2 \cup \cdots \cup D_n \) be an oriented virtual link diagram, with components \( D_1, \ldots, D_n \), such that
	\begin{equation*}
	lk ( D_i, D_j ) \in 2 \Z
	\end{equation*}
	for all \( i, j \in \lbrace 1, \ldots , n \rbrace \).
	
	A self-crossing of \( D \) is \emph{IP-even} if the associated chord of \( G ( D ) \) encloses an even number of chord endpoints (this is well-defined as \( D \) has even pairwise linking numbers). Otherwise it is \emph{IP-odd}.
	
	The Gauss code \( G ( D ) \) is further investigated, to determine whether or not it satisfies a certain condition. There are two possible outcomes:
		\begin{enumerate}
			\item every mixed classical crossing of \( D \) is declared as IP-odd
			\item every mixed classical crossing of \( D \) is declared as IP-even
		\end{enumerate}
	Which case holds depends on whether or not \( G ( D ) \) satisfies the particular condition. For full details see \cite[Definition \(2.1\)]{ImPark2013}.
\end{definition}

Im and Park show that this definition yields a parity on the set of virtual links with even pairwise linking numbers. To be precise, they erroneously claim that it is a parity on arbitrary virtual links. However, they present an example, given in \cite[Figure 7]{ImPark2013}, demonstrating that the construction does not satisfy the third parity axiom (as given in \Cref{2Def:parityaxioms}). This error can be sourced to the proof of Lemma 2.4 on page \(4 \) of \cite{ImPark2013}: in it, Im and Park assume that if \( D \) and \(D'\) are virtual link diagrams related by a Reidemeister \(3 \) move, then \( D \) and \( D'\) have even pairwise linking numbers. This is false in general, as \cite[Figure 7]{ImPark2013} shows. As a consequence, \Cref{5Def:ipparity} yields a parity on the class of virtual links with even pairwise linking numbers.

The associated writhe invariant is defined as follows.

\begin{definition}[IP-writhe]\label{5Def:ipwrithe}
	Let \( D \) be a diagram of an oriented virtual link \( L \) with even pairwise linking numbers. Define the \emph{IP-writhe} as
		\begin{equation*}
			IP ( D ) \coloneqq \sum_{ c ~\text{IP-odd}} \text{sign}( c ).
		\end{equation*}
	The invariance of \( IP(D) \) under the virtual Reidemeister moves follows from the fact that the IP parity satisfies the parity axioms on such links, and we may define the IP-writhe of \( L \) as \( IP ( L ) \coloneqq IP ( D ) \).
\end{definition}

We shall also need the following restricted invariant.

\begin{definition}\label{5Def:selfipwrithe}
	Let \( D \) be a diagram of an oriented virtual link \( L \) with even pairwise linking numbers. Define the following quantity
		\begin{equation*}
			IP_S ( D ) \coloneqq \sum_{\substack{c ~\text{IP-odd} \\ \text{self-crossing}}} \text{sign}( c ).
		\end{equation*}
	Again, the invariance of \( IP(D) \) under the virtual Reidemeister moves follows from the fact that the IP parity satisfies the parity axioms on such links, and we may define \( IP_S ( L ) \coloneqq IP_S ( D ) \).
\end{definition}

It follows from \Cref{5Def:ipparity} that a mixed classical crossing between the components \( D_i \) and \( D_j \) is IP-odd if and only if all such crossings are IP-odd. From this observation we obtain
\begin{equation}\label{5Eq:ipwrithe2}
	IP ( D ) = \left\lbrace \begin{matrix}
				\sum_{i} \left( \sum_{ j > i } lk ( D_i , D_j ) \right) + IP_S ( D ), ~\text{or} \\
				IP_S ( D )
				\end{matrix}\right.
\end{equation}

As a consequence of \Cref{5Eq:ipwrithe2} the contribution of the mixed crossings to the IP writhe is determined by the pairwise linking numbers of the argument virtual link. As in the case of the na\"\i ve parity, this motivates the following theorem.

\begin{theorem}\label{5Thm:2colourstronger}
	The pair \( \left( J^2 ( L ), J^2_S ( L ) \right) \) forms a strictly stronger invariant than the pair \( \left( IP ( L ), IP_S ( L ) \right) \) on the set of oriented virtual links with even pairwise linking numbers (for \( J^2_S ( L ) \) given in \Cref{3Def:2colourselfwrithe}).
\end{theorem}

\begin{proof}
	We show that if \(L\) is an oriented virtual link with even pairwise linking numbers such that \( \left( J^2 ( L ), J^2_S ( L ) \right) \) is trivial, then \( \left( IP ( L ), IP_S ( L ) \right) \) is trivial also, and that there exist virtual links detected by \( \left( J^2 ( L ), J^2_S ( L ) \right) \) but not \( \left( IP ( L ), IP_S ( L ) \right) \).
	
	Given a diagram \( D = D_1 \cup D_2 \cup \cdots \cup D_n \), repeat the first part of the proof of \Cref{5Thm:2colournwrithe} to obtain
	\begin{equation}\label{5Eq:linkingnumber2}
		\begin{aligned}
			lk ( D_i, D_j ) &= \Sigma \left( C^i \cap C^j \cap \mathscr{C}^o \right) + \Sigma \left( C^i \cap C^j \cap \mathscr{C}^e \right) \\
							&= 2 \left( \Sigma \left( C^i \cap C^j \cap \mathscr{C}^o \right) \right).
		\end{aligned}
	\end{equation}
	It follows that \( lk (D_i, D_j) \in 2 \Z \). As a consequence, given a \(2\)-colouring \( \mathscr{C} \) a self-crossing of \( D \) is odd with respect to \( \cp_{\mathscr{C}} \) if and only if it is IP-odd (this may be seen by repeating the argument given in the proof of \Cref{2Prop:Gparityrelation}). This implies that
	\begin{equation}\label{5Eq:selfwrithes}
		IP_S ( L ) = J^2_{S} ( L ) = 0
	\end{equation}
	(recall that \( J^2_{S} ( L ) \) is trivial by hypothesis).
	
	Next, fix a \(2\)-colouring \( \mathscr{C} \) and repeat the second part of the proof of \Cref{5Thm:2colournwrithe} to obtain
	\begin{equation*}
		\sum_{j>i} \left( \Sigma \left( C^i \cap C^j \cap \mathscr{C}^o \right) \right) = 0
	\end{equation*}
	from which we observe
	\begin{equation*}
		\begin{aligned}
			0 &= \sum_i \left(\sum_{j > i} \left( \Sigma \left( C^i \cap C^j \cap \mathscr{C}^o \right) \right) \right) \\
			&= \frac{1}{2} \sum_i \left( \sum_{j > i} lk ( D_i , D_j ) \right)\\
			&=\frac{1}{2} IP ( L )
		\end{aligned}
	\end{equation*}
	by \Cref{5Eq:ipwrithe2,,5Eq:linkingnumber2,,5Eq:selfwrithes} (notice that the argument proceeds irrespective of which case holds in \Cref{5Eq:ipwrithe2}), and we may conclude that \( IP (  L ) = 0 \).
	
	Im and Park show that \( \left( IP ( L ), IP_S ( L ) \right) \) is trivial on chequerboard colourable virtual links \cite[Proposition \(2.11\)]{ImPark2013} (they use the term \emph{normal} for such links). As all classical links are chequerboard colourable, it follows that \( \left( IP ( L ), IP_S ( L ) \right) \) is trivial on classical links also. We conclude by observing that there are many classical links with non-trivial \( \left( J^2 ( L ), J^2_S ( L ) \right) \); a simple example is given in \Cref{5Fig:hopf}. The \(4\)-component link depicted in \Cref{5Fig:cbdetected} is also an example of a virtual link detected by the \(2\)-colour writhe but not the IP writhe.
\end{proof}

In fact, there are virtual links that are not detected by the IP writhe, but detected by the na\"\i ve writhe: again \Cref{5Fig:hopf} provides an example. It is also important to note that while the \(2\)-colour writhe of a chequerboard colourable virtual link can be determined from the pairwise linking numbers, the \(2\)-colour writhe remains strictly stronger than both the na\"\i ve writhe and the IP writhe on this class of links. This may be seen by noticing that the link depicted in \Cref{5Fig:cbdetected} is chequerboard colourable, and is detected only by the \(2\)-colour writhe.

\begin{figure}
	\includegraphics[scale=1]{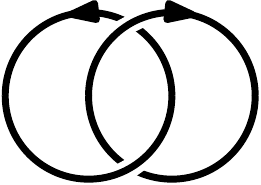}
	\caption{A classical Hopf link that is detected by the \(2\)-colour writhe and the na\"\i ve writhe, but is not detected by the IP writhe.}
	\label{5Fig:hopf}
\end{figure}

The IP parity has been used to construct a number of polynomial invariants of virtual links \cite{ImLeeLee2014,ImKimPark2017}. In light of \Cref{5Thm:2colourstronger} it is reasonable to suspect that polynomial invariants constructed using the \(2\)-colour parity will be stronger than those constructed using the IP parity.

\subsection{Other parities}\label{5Subsec:other}
We conclude by comparing the \(2\)-colour parity to other constructions related to parity theories of virtual links.

\subsubsection{Manturov's parity of free knots}\label{5Subsub:manturov}
Free links are a drastic simplification of virtual links, obtained by considering virtual link diagrams up to the virtual Reidemeister moves, classical crossing changes, and a further move known as flanking. Alternatively, free links may be defined in terms of the simple Gauss diagrams employed in \Cref{2Sec:definition}. The reader familiar with free links may therefore suspect that the \(2\)-colour parity descends to a parity of free links. Here we confirm these suspicions.

\begin{definition}
	\label{5Def:freelinks} A \emph{free link} is an equivalence class of simple Gauss diagrams, up to the following local moves:
		\begin{equation*}
		\renewcommand*{\arraystretch}{2}
			\begin{matrix}
				\includegraphics[scale=0.65]{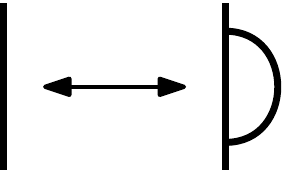} & \includegraphics[scale=0.65]{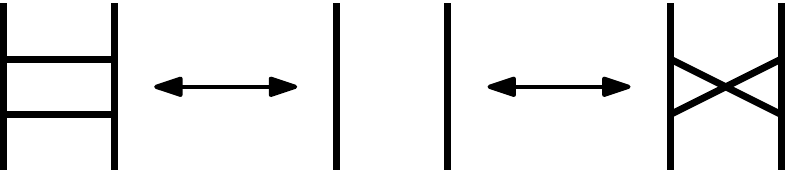} \\
				\includegraphics[scale=0.65]{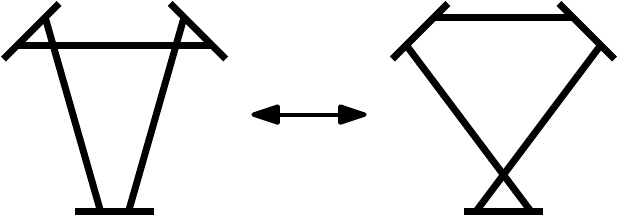} & \includegraphics[scale=0.65]{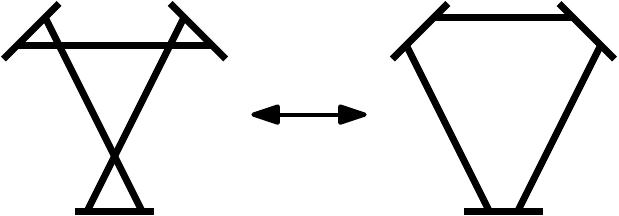} \\
			\end{matrix}
		\end{equation*}
	where open line segments depict neighbourhoods of the chord endpoints in the core circles.
\end{definition}

A free link is \(2\)-colourable if it has a representative that is \(2\)-colourable as a simple Gauss diagram (as in \Cref{2Def:altgausscolour}).

\begin{proposition}\label{5Prop:freelinks}
	The \(2\)-colour parity descends to a parity on \(2\)-colourable oriented free links.
\end{proposition}

\begin{proof}
	First, we show that the \(2\)-colour parity may be determined directly from a simple Gauss diagram. Let \( G \) be a \(2\)-colourable oriented simple Gauss diagram  (as in \Cref{2Def:altgausscolour}). Given a \( \mathscr{C} \) a \(2\)-colouring of \( G \), define a parity, \( \cp_{\mathscr{C}} \), on the chords of \( G \) in the following manner:
	\begin{equation*}
		\begin{aligned}
			&\cp_{\mathscr{C}} \left( \raisebox{-14pt}{\includegraphics[scale=0.65]{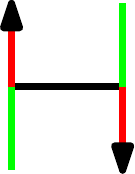}} \right) = \cp_{\mathscr{C}} \left( \raisebox{-14pt}{\includegraphics[scale=0.65]{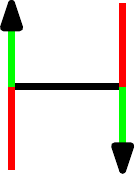}} \right) = 1 \\
			&\cp_{\mathscr{C}} \left( \raisebox{-14pt}{\includegraphics[scale=0.65]{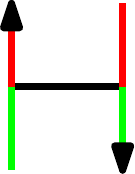}} \right) = \cp_{\mathscr{C}} \left( \raisebox{-14pt}{\includegraphics[scale=0.65]{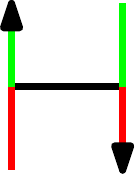}} \right) = 0.
		\end{aligned}
	\end{equation*}
	It is easy to see that this designation is equivalent to that given in \Cref{2Eq:oddeven}, and that it satisfies the parity axioms with respect to the moves given in \Cref{5Def:freelinks}.
\end{proof}

In \cite{Manturov2012} Manturov defines a parity of free knots, and extends it to free links appearing in concordances between free knots (\cite{Manturov2012} also contains the relevant definitions regarding free knot concordance). Further, in \cite{Manturov2016} an extension of parity to \(2\)-component virtual links is given. The \(2\)-colour parity has the advantage that it may be computed directly from a simple Gauss diagram, without requiring additional concordance information.

While the \(2\)-colour parity descends to free links, the \(2\)-colour writhe does not. For example, the virtual knot given in \Cref{2Fig:knot21} has non-trivial \(2\)-colour writhe, but represents a trivial free knot.

\subsubsection{The affine index polynomial}\label{5Subsub:affine} The affine index polynomial is an invariant of virtual knots due to Kauffman \cite{Kauffman2013}, which is related to the Gaussian parity. In \cite{Kauffman2018} Kauffman demonstrates that the affine index polynomial is a concordance invariant, and determines the class of virtual links to which it extends. Kauffman refers to such virtual links as \emph{compatible}. We demonstrate that the set of compatible virtual links is a proper subset of that of \( 2 \)-colourable virtual links.

\begin{definition}[\cite{Kauffman2018}]
	Let \( D \) be an oriented virtual link diagram. We say that \( D \) is \emph{compatible} if every component of \( D \) has algebraic intersection number zero with the remainder of \( D \).
\end{definition}
(Here algebraic intersection should be understood as signed intersection in the plane.)

\begin{proposition}
	A compatible virtual link is \(2\)-colourable.
\end{proposition}

\begin{proof}
	We prove the contrapositive. Let \( D \) have a degenerate component (as in \Cref{2Def:degen}). An odd number of chord endpoints lie on this component, so that there must be an odd number of mixed crossings between it and the remainder of the diagram. It follows immediately that this component cannot have algebraic intersection zero with the remainder of the diagram.
\end{proof}

\begin{figure}
	\includegraphics[scale=0.75]{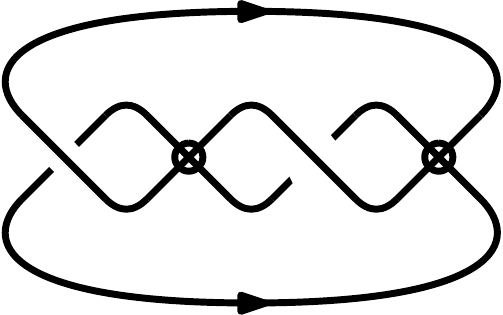}
	\caption{A \(2\)-colourable virtual link diagram that is not compatible.}
	\label{5Fig:unlabelable}
\end{figure}

\Cref{5Fig:unlabelable} depicts a \(2\)-colourable virtual link that is not compatible, from which it follows that the set of compatible virtual links is a proper subset of that of \( 2 \)-colourable virtual links.

\subsubsection{Xu's index polynomial}\label{5Subsub:xu} In \cite{Xu2018} Xu defines an index polynomial of virtual links, using an index theory related to the IP parity. The definition of this polynomial also suffers from the defect of handling self- and mixed crossings differently. As a consequence, it cannot detect virtual links possessing a diagram with no self crossings, a blind-spot that the \(2\)-colour writhe does not possess. Examples of a virtual links not detected by Xu's index polynomial, but detected by the \(2\)-colour writhe are given in \Cref{3Fig:example1,5Fig:cbdetected,5Fig:unlabelable}.

\bibliographystyle{plain}
\bibliography{library}

\end{document}